\documentclass[11pt]{amsart}
\usepackage{lipsum}
\usepackage{verbatim}
\usepackage{enumerate}
\usepackage[margin=1.2in]{geometry}
\usepackage{mathtools}
\usepackage[colorlinks]{hyperref}
\usepackage{amsmath}

\usepackage{changepage}
\usepackage{amsmath}
\usepackage{amssymb}
\usepackage{lipsum}

\usepackage{pgfpages}

\usepackage{amsmath}
\usepackage{graphicx}
\usepackage{eucal}
\usepackage{dsfont}
\usepackage{tikz}
\usetikzlibrary{cd}
\def\presuper#1#2%
  {\mathop{}%
   \mathopen{\vphantom{#2}}^{#1}%
   \kern-\scriptspace%
   #2}

\def\Z{\mathbb{Z}}

\def\F{\mathrm{F}}
\def\ind{\mathrm{ind}}
\def\W{\mathrm{W}}
\def\G{\mathrm{G}}

\def\R{\mathrm{R}}

\def\diag{\mathrm{diag}}

\theoremstyle{claim}

\newtheorem{thm}{Theorem}[section]

\newtheorem{proposition}[thm]{Proposition}

\newtheorem{lemma}[thm]{Lemma}
\newtheorem{cor}[thm]{Corollary}

\newtheorem{definition}[thm]{Definition}
\newtheorem{remark}[thm]{Remark}

\newtheorem*{propsccond}{{Proposition \ref{Scconductor}}}
\newtheorem*{theorem}{Theorem}

\newcommand{\GL}{\mathrm{GL}}

\newcommand{\U}{\mathrm{U}}

\newcommand{\St}{\mathrm{St}}
\newcommand{\End}{\mathrm{End}}
\newcommand{\Ql}{\overline{\mathbb{Q}_\ell}}
\newcommand{\Fl}{\overline{\mathbb{F}_\ell}}

\newcommand{\Res}{\mathbf{Res}}
\newcommand{\Ind}{\mathrm{Ind}}

\newcommand{\Hom}{\mathbf{Hom}}

\def\P{\mathrm{P}}
\def\N{\mathrm{N}}
\def\M{\mathrm{M}}
\usetikzlibrary{arrows,positioning,shapes}
\tikzset{
    >=stealth',
    punkt/.style={
           rectangle,
           rounded corners,
           draw=black, very thick,
           text width=11em,
           minimum height=2em,
           text centered},
    punkt1/.style={
           rectangle,
           rounded corners,
           draw=black, 
           text width=11em,
           minimum height=2em,
           text centered},
    pil/.style={
           ->,
           thick,
           shorten <=2pt,
           shorten >=2pt,}
}
\tikzstyle{every picture}+=[remember picture]
\date{}

\def\U{\mathrm{U}}

\def\Hom{\mathrm{Hom}}
\def\End{\mathrm{End}}
\def\Res{\mathrm{Res}}

\def\ind{\mathrm{ind}}

\def\Ql{\overline{\mathbb{Q}_\ell}}

\def\St{\mathrm{St}}

\def\V{\mathrm{V}}

\def\A{\mathrm{A}}
\def\B{\mathrm{B}}

\def\F{\mathrm{F}}
\def\E{\mathrm{E}}
\def\G{\mathrm{G}}
\def\H{\mathrm{H}}

\def\J{\mathrm{J}}
\def\K{\mathrm{K}}

\def\M{\mathrm{M}}
\def\P{\mathrm{P}}
\def\R{\mathrm{R}}
\def\S{\mathrm{S}}
\def\T{\mathrm{T}}
\def\U{\mathrm{U}}

\def\W{\mathrm{W}}
\def\Z{\mathrm{Z}}

\def\Ind{\mathrm{Ind}}

\def\op{\mathrm{op}}

\newcommand{\margh}[1]{}

\def\Fl{\overline{\mathbb{F}_\ell}}
\def\Zl{\overline{\mathbb{Z}_\ell}}

\def\GL{\mathrm{GL}}
\def\N{\mathrm{N}}

\def\new{\mathrm{new}}

	\title{Newforms in cuspidal representations}
	\author{Johannes Girsch}
	\address{Johannes Girsch, School of Mathematics and Statistics, University of Sheffield, Sheffield, S3 7RH, United Kingdom.}
\email{j.girsch@sheffield.ac.uk}

	\author{Robert Kurinczuk}
	\address{Robert Kurinczuk, School of Mathematics and Statistics, University of Sheffield, Sheffield, S3 7RH, United Kingdom.}
\email{robkurinczuk@gmail.com}
	
\begin{document} 
	
	\maketitle
\begin{abstract}
We consider newform vectors in cuspidal representations of $p$-adic general linear groups.  We extend the theory from the complex setting to include~$\ell$-modular representations with~$\ell\neq p$, and prove that the conductor is compatible with congruences modulo~$\ell$ for (ramified) supercuspidal~$\ell$-modular representations and for depth zero cuspidals. In the complex and modular setting, we prove explicit formulae for depth zero and minimax cuspidal representations of integral depth, in Bushnell-Kutzko and Whittaker models.
\end{abstract}
	\setcounter{tocdepth}{1}
	\tableofcontents
	\section{Introduction}

\subsubsection*{Newforms}	
Casselman \cite{Casselman} adelized the classical theory of newforms, translating them into the language of automorphic representations for~$\GL(2)$.  Jacquet--Piatetski-Shapiro--Shalika \cite{JPSS, Jacquet} and Matringe \cite{Matringe} generalized this picture to~$\GL(n)$, showing the existence and uniqueness of newforms for generic~$\mathbb{C}$-representations of~$\G_n=\GL_n(\F)$, where~$\F$ is a non-archimedean local field of residue characteristic~$p$.  (For archimedean newform theory see \cite{humphries2024archimedean}.) Jacquet--Piatetski-Shapiro--Shalika and Matringe consider a decreasing family of compact open subgroups $(\K_n(m))_{m\in\mathbb{N}}$ of~$\G_n$, whose intersection consists of matrices in~$\GL_n(\mathfrak{o}_\F)$ with final row~$\begin{pmatrix}0&\cdots&0&1\end{pmatrix}$, where~$\mathfrak{o}_\F$ denotes the ring of integers of~$\F$.    Amazingly, it turns out that for~$\pi$ a ramified generic~$\mathbb{C}$-representation of~$\G_n$, there exists a (unique) positive integer~$c(\pi)$ such that
\begin{equation*}
\pi^{\K_n(m)}\simeq\begin{cases}
\mathbb{C}&\text{if } m= c(\pi);\\
0&\text{if }m<c(\pi).\end{cases}\end{equation*}
A non-zero vector of~$\pi^{\K_n(c(\pi))}$ is called a \emph{newform}.   Locally, newforms and their realizations in Whittaker models have proven useful as test vectors; see, for example, \cite{JPSS}, \cite{MR3719298}, \cite{MR4553601}, \cite{MR4224513}.

\subsubsection*{Cuspidal representations}
Let~$\R$ now denote an algebraically closed field of characteristic~$\ell\neq p$.  Harish-Chandra's approach to classifying the irreducible~$\R$-representations begins by considering the \emph{cuspidal~$\R$-representations} (resp.~\emph{supercuspidal} $\R$-representations) that is the irreducible~$\R$-representations which do not appear as a quotient (resp.~subquotient) of an~$\R$-representation parabolically induced from an irreducible~$\R$-representation of a proper parabolic subgroup.  Suprisingly, by work of Bushnell--Kutzko \cite{BK93} extended to the modular setting by Vign\'eras \cite{Vig96}, it turns out that all cuspidal~$\R$-representations of~$\G_n$ are compactly induced: for every cuspidal~$\R$-representation~$\pi$ of~$\G_n$ there exists an (explicitly constructed) pair~$(\mathbf{J},{\boldsymbol\lambda})$, consisting of a compact-mod-centre subgroup~$\mathbf{J}$ of~$\G_n$ and an irreducible representation~${\boldsymbol\lambda}$ of~$\mathbf{J}$, in Bushnell--Kutzko's list such that~$\pi\simeq \ind_{\mathbf{J}}^{\G_n}({\boldsymbol\lambda})$.  

Paskunas--Stevens \cite{PaskSte} use Bushnell and Kutzko's models to construct explicit Whittaker functions with small support in the Whittaker model of a cuspidal representation of~$\G_n$.  These functions have also proved useful as test vectors; see for example \cite{MR3901635}, \cite{AKMSS}.

\subsubsection*{}
These two pictures lead to the following natural questions which we discuss in the paper:
\begin{enumerate}
\item\label{Q1} Describe the (unique up to scalar) newform in the Bushnell--Kutzko model of a cuspidal~$\mathbb{C}$-representation of~$\G_n$.
\item\label{Q2} Compare the local test vectors arising from the globally motivated newform theory and the locally motivated Bushnell--Kutzko theory.
\item\label{Q3} Does the theory of newforms extend to~$\R$-representations?
\end{enumerate}

\subsubsection*{Our results}
We approach these questions explicitly in the special cases of \emph{depth zero} and \emph{minimax} cuspidal representations.  Note that this infinite family is reasonably broad, for example every cuspidal representation~$\pi$ of~$\G_k$ for~$k$ prime which is \emph{twist minimal} (i.e., it is of minimal depth among the family of representations obtained from~$\pi$ by twisting by a character) is either minimax or depth zero.

For question \eqref{Q3}, the existence of newforms in a cuspidal~$\R$-representation~$\pi$ follows by reduction modulo~$\ell$ (note that for a cuspidal~$\R$-representation over an algebraically closed field of characteristic~$\ell$, there exists an unramified twist of it defined over~$\Fl$).  By a lifting argument we first show: 
\begin{propsccond}
Let~$\pi$ be an integral supercuspidal~$\Ql$-representation of $G_n$ for $n\geqslant 2$, which has supercuspidal reduction modulo~$\ell$, then we have an equality of \emph{conductors}~$c(\pi)=c(r_{\ell}(\pi))$. 

\end{propsccond}
We later extend this to include all integral depth zero supercuspidal~$\Ql$-representations.  For uniqueness of newforms in positive characteristic, we prove this explicitly using a Mackey theory and lifting argument for depth zero cuspidal~$\Fl$-representations and unramified minimax cuspidal~$\Fl$-representations.

Question \eqref{Q1} was first considered for cuspidal~$\mathbb{C}$-representations of~$\G_n$ of depth zero by Reeder \cite{Reeder}, where he shows in the model of a depth zero cuspidal compactly induced from~$\Z_n \K_n$, where~$\K_n=\GL_n(\mathfrak{o}_\F)$ and $\Z_n$ is the center of $G_n$, that a newform vector has support in~$\Z_n\K_n\Sigma_n\K_n(c(\pi))$, where\[\Sigma_n=\diag\left(\varpi_\F^{n-1},\varpi_\F^{n-2},\ldots, 1\right).\]
We give a new proof of Reeder's result which works more generally for depth zero cuspidal~$\R$-representations (also establishing existence and uniqueness of newforms directly for depth zero cuspidals~$\R$-representations without using Jacquet--Piatetski-Shapiro--Shalika's or Matringe's work), and use this to show the newform is an average of the Bessel vector considered by Paskunas and Stevens.  Using this expression we obtain formulae for the matrix coefficients and Whittaker functions associated to newforms in depth zero cuspidal~$\R$-representations, in particular showing they are an average of the corresponding Paskunas--Stevens vectors giving an answer to \eqref{Q2} in this case.  

For a cuspidal representation~$\tau$ of~$\GL_n(k_\F)$ we write~$B_{\tau,\overline{\psi}}$ for its associated Bessel function (see Section \ref{Section:BesselFunctions}).  Our main depth zero theorem is:

\begin{theorem}
[{Theorem \ref{depthzerocoset}, Propositions \ref{matrixcoeffsdepthzero} and \ref{depthzerowhittakernewform}}]
Let~$n\geqslant 2$, and~$\pi$ be a depth zero cuspidal~$\R$-representation containing the cuspidal~$\R$-type~$(\K_n,\tau)$.
\begin{enumerate}
\item \begin{enumerate}
\item (existence, uniqueness, and support of newforms) the conductor is given by~$c(\pi)=n$,
\[\Hom_{\R[\K_n(c(\pi))]}(1,\pi)\simeq \R,\] and the unique up to scalar (non-zero newform)~$f_{\new}\in\pi^{\K_n(n)}$ has support
\[\mathrm{supp}(f_{\new})\subseteq \Z_n\K_n\Sigma_n\K_n(n).\]
\item  (explicit formula in terms of Bessel functions) The function~$f_{\new}\in \pi^{\K_n(n)}$ is characterized by its support and
\[ f_{\new}(\Sigma_n)=\sum_{b\in \B_{n-1}^{\op}(k_\F)}b\cdot B_{\tau,\overline{\psi}}.\]
Moreover, there exists a unique~$\R$-Haar measure on~$\K_n(n)$ such that, for all~$g\in\G$, 
\[f_{\new}(g)=\int_{\K_n(n)} \mathcal{B}_{\widetilde{\tau}}(gk\Sigma_n^{-1}) dk,\]
where $\mathcal{B}_{\widetilde{\tau}}$ is the Bessel vector in $\ind_{\Z_n\K_n}^{\G_n}(\widetilde{\tau})$ (see Section \ref{Section:depthzero}). 
\end{enumerate}
\item (Depth zero matrix coefficients of newforms)
Let $y_1,\dotsc,y_r$ be a set of coset representatives of $\K_n\Z_n\backslash \K_n\Z_n\Sigma_n \K_n(n)$.  The matrix coefficient~$c_{f_{\new},f_{\new}^{\vee}}$ is non-zero, bi-$\K_n(n)$-invariant, has support $\mathop{supp}(c_{f_{\new},f_{\new}^{\vee}})\subseteq \Z_n\K_n(n)\Sigma_n^{-1}\K_n\Sigma_n\K_n(n)$ and for~$g\in \Sigma_n^{-1}\K_n\Sigma_n$ we have
	\[c_{f_\new,f_{new}^\vee}(g)=\sum_{(i,j)\in I_g}\frac{|\G(\mathbb{F}_q)|}{|\U(\mathbb{F}_q)|\dim(\tau)}\sum_{b,b'\in \B_{n-1}^{\op}(\mathbb{F}_q)}B_{\tau,\overline{\psi}}(b\Sigma_n y_jgy_i^{-1}\Sigma_n^{-1}b'),\]
	where~$I_{g}=\{(i,j)\in\mathbb Z^2\mid 1\leqslant i,j\leqslant r\text{ and }g\in y_j^{-1}\Z_n\Sigma_n^{-1}\K_n\Sigma_ny_i\}$.
\item (Depth zero Whittaker newforms) Suppose~$\psi:\F\rightarrow \mathbb{C}^\times$ has conductor~$\mathfrak{o}_\F$, denote by~$\psi$ its extension to a non-degenerate character of the upper triangular unipotent subgroup~$\U_n$ of~$\G_n$ via precomposing with $u\mapsto \sum_{i=1}^{n-1}u_{i,i+1}$, and let~$W_{\pi,\new,\psi}$ denote the Whittaker newform~of~$\pi$ (normalized at the identity).  Then 
\[\mathrm{supp}(\W_{\pi,\new,\psi})\subseteq \Z_n\U_n\Sigma_n^{-1}\K_n\Sigma_n\K_n(n),\]
and for~$g\in\G$, %there is an~$\R$-Haar measure on~$\K_n(n)$ such that
\[ \W_{\pi,\new,\psi}(g)=\int_{\K_n(n)}  \W_{\pi,\mathrm{Gel},\psi^{\Sigma_n^{-1}}}(\Sigma_n g k\Sigma_n^{-1})  dk\]
for an appropriately normalized~$\R$-Haar measure~$dk$ on~$\K_n(n)$, where~$ \W_{\pi,\mathrm{Gel},\psi^{\Sigma_n^{-1}}}$ denotes Gelfand's explicit Whittaker function for~$\pi$ (cf.~\cite{PaskSte}).  
\end{enumerate}
\end{theorem}

A \emph{minimax cuspidal~$\R$-representation}~$\pi$ of~$\G_n$ has a Bushnell--Kutzko model as an induced representation from a compact-mod-centre subgroup~$\mathbf{J}=\E^\times \U^{\lfloor (m+1)/2\rfloor}(\Lambda)$ where~$\E/\F$ is a degree~$n$ field extension embedded in~$\M_n(\F)$, and~$\U^{\lfloor (m+1)/2\rfloor}(\Lambda)$ is a filtration subgroup of a parahoric subgroup of~$\G_n$ associated to the point~$\Lambda$ in the building of~$\G$ (for some specific~$m$ given by the input data for~$\pi$, see Section \ref{ss:strata}).  

One invariant attached to~$\pi$ is its \emph{depth}, in terms of the above inducing data this is given by~$m/e(\E/\F)$.  The special case of minimal depth $m=1$, $e(\E/\F)=n$ the representations are called simple supercuspidals and questions \ref{Q1} and \ref{Q2} are answered in \cite{simpleSCs}.  

Our main minimax theorem on explicit vectors is Theorem \ref{Propositionunramnewforms}, this gives an analogue of the first part of our main depth zero theorem for minimax cuspidal representations of integral depth.  We obtain from this an expression for the associated Whittaker newform showing it is an average of the explicit Whittaker vector of Paskunas--Stevens, Corollary \ref{WhittcorollaryURminimax}.

In further future work, we consider the non-cuspidal case of newforms in generic $\ell$-modular representations.

\subsubsection*{Acknowledgements} Both authors were supported by EPSRC grant EP/V001930/1, and the second author was supported by the Heilbronn Institute for Mathematical Research.  We thank Nadir Matringe and Shaun Stevens for useful conversations.

	\section{Notation}
	\subsection{Smooth representations}
	Let~$\F$ be a non-archimedean local field, with ring of integers~$\mathfrak{o}_{\F}$.  Let~$\mathfrak{p}_{\F}=(\varpi_{\F})$ denote the unique maximal ideal of~$\mathfrak{o}_{\F}$, and~$k_{\F}=\mathfrak{o}_{\F}/\mathfrak{p}_{\F}$ the residue field -- a finite field of size~$q$ a power of~$p$. 
	
Let~$\G_n=\GL_n(\F)$,~$\K_n=\GL_n(\mathfrak{o}_\F)$ and $\Z_n\simeq \F^\times$ denotes the centre of~$\G_n$.  We drop the subscript~$_n$ when it is clear and write~$\G=\GL_n(\F)$,~$\K=\K_n$ etc.

	Let~$\R$ be an algebraically closed field of characteristic~$\ell\neq p$.  Let~$\H$ be a locally profinite group. By an~$\R$-representation of~$\H$ we mean a smooth representation of~$\H$ on an~$\R$-vector space.  
	
	For a closed subgroup~$\J$ of~$\H$, a smooth~$\R$-representation~$(\pi,\mathcal{V})$ of~$\J$, and~$h\in\H$, we write
	\begin{enumerate}[-]
	\item $\pi^h$ for the smooth representation~$(\pi^h,\mathcal{V})$ of~$\J^h:=h^{-1}\J h$, where~$\pi^h:h^{-1}jh\mapsto \pi(j)$.
	\item $\presuper{h}\pi$ for the smooth representation~$(\presuper{h}\pi,\mathcal{V})$ of~$\presuper{h}\J:=h\J h^{-1}$, where~$\pi^h:hjh^{-1}\mapsto \pi(j)$.
	\end{enumerate}
	
	For any commutative ring $\S$ and positive integers $i,j$ let $\M_{i\times j}(\S)$ be the ring of $i\times j$ matrices with entries in $\S$. We will write $1_j$ for the identity matrix in $\M_{j\times j}(\S)$.
	
\subsection{Bessel functions}\label{Section:BesselFunctions}
Let~$\tau$ be a cuspidal~$\R$-representation of~$\GL_n(k_\F)$.  Let~$\B(k_\F)$ be the standard Borel subgroup of~$\GL_n(k_\F)$ of upper triangular matrices, with unipotent radical~$\U(k_\F)$.  Fix a non-trivial character~$\overline{\psi}:k_{\F}\rightarrow \R^\times$, which we extend to a non-degenerate character~$\overline{\psi}:\U(k_\F)\rightarrow \R^\times$ defined by~$\overline{\psi}(u)=\overline{\psi}(u_{1,2}+\cdots+u_{n-1,n})$.

Then~$\tau$ is \emph{generic}:
\[\Hom_{\R[\GL_n(k_\F)]}(\tau,\ind_{\U(k_\F)}^{\GL_n(k_\F)}(\overline{\psi}))\simeq\R,\]
and we write~$\W(\tau,\overline{\psi})$ for its Whittaker model (that is the image of any non-zero morphism in~$\ind_{\U(k_\F)}^{\GL_n(k_\F)}(\overline{\psi}))$).

The Bessel function~$B_{\tau,\overline{\psi}}\in W(\tau,\overline{\psi})$ is the unique~$\overline{\psi}$-bi-invariant function in~$W(\tau,\overline{\psi})$ with~$B_{\tau,\overline{\psi}}(1)=1$.  We record the following useful properties of~$B_{\tau,\overline{\psi}}$ which follow from \cite{Gelfand} and the compatibility of the Bessel function with reduction modulo~$\ell$:

\begin{enumerate}[(B1)]
\item \label{besselfunction1} Let~$\P_n(k_\F)$ be the standard mirabolic subgroup of~$\GL_n(k_\F)$ of all matrices with last row~$\left(\begin{smallmatrix}0&\cdots &0&1\end{smallmatrix}\right)$.  Then for~$p\in\P_n(k_\F)$, the function~$B_{\tau,\overline{\psi}}(p)$ is nonzero if and only if $p\in\U(k_\F)$.
\item\label{besselfunction2} Let~$\chi_{\tau}$ denote the trace character of~$\tau$.  Then~$B_{\tau,\overline{\psi}}(a)=|\U(\mathbb{F}_q)|^{-1}\sum_{\U(\mathbb{F}_q)}\overline{\psi}^{-1}(u)\chi_{\tau}(au)$, for~$a\in\GL_n(k_\F)$.
\item For~$z\in k_{\F}^\times$ and~$a\in\GL_n(k_\F)$, we have~$B_{\tau,\overline{\psi}}(az)=\omega_{\tau}(z)B_{\tau,\overline{\psi}}(a)$.
\item Let~$\tau^\vee$ denote the contragredient of~$\tau$, for~$a\in \GL_n(k_\F)$, we have $B_{\tau,\overline{\psi}}(a^{-1})=B_{\tau^{\vee},\overline{\psi}^{-1}}(a)$.
\end{enumerate}

	\section{Newforms for cuspidal representations}

\subsection{Conductors and newforms}

\begin{definition}
Let~$\pi$ be a generic~$\R$-representation of~$\G_n$ with~$n\geqslant 2$.
\begin{enumerate}
\item For~$m\in\mathbb{N}$, let~$\K_n(m)$ denote the \emph{conductor subgroup} of~$\G_n$, defined by
\[\K_n(m)=\left\{\left(\begin{smallmatrix}a&b\\c&d\end{smallmatrix}\right)\in \K_n:c\in \M_{1\times (n-1)}(\mathfrak{p}_{\F}^m), d\in (1+\mathfrak{p}_\F^m)\right\}.\]
Set~$\K_n(0)=\K_n=\GL_n(\mathfrak{o}_\F)$.
\item If there is a nonnegative integer such that~$\pi^{\K_n(m)}\neq 0$, then we define the \emph{conductor} $c(\pi)$ of~$\pi$ to be the minimal (non-negative)~$m$ such that~$\pi^{\K_n(m)}\neq 0$.
\item We say that~$\pi$ is \emph{unramified} if~$c(\pi)=0$ and \emph{ramified} otherwise.
\item Let~$\mathcal{V}$ be the underlying~$\R$-vector space of~$\pi$ considered as an~$\R[\G_n]$-module via~$\pi$.  We call an element of~$\mathcal{V}^{\K_n(c(\pi))}$ a newform for~$\pi$.
\end{enumerate}
\end{definition}

The existence of newforms for generic $\mathbb{C}$-representations (or equivalently~$\Ql$-representations) is established by Jacquet--Piatetski-Shapiro--Shalika, with a different proof given by Matringe:

\begin{thm}[{\cite{JPSS,Jacquet,Matringe}}]\label{JPSSMatringe}
Let~$\pi$ be a generic~$\Ql$-representation of~$\G_n$.  The conductor~$c(\pi)$ exists and, moreover,~$\pi^{\K_n(c(\pi))}$ is one-dimensional.
\end{thm}

Let~$\pi$ be a cuspidal~$\Ql$-representation of~$\G_n$ with~$n\geqslant 2$, then by \cite{Bushnell},~
\begin{equation}\tag{$\star$}
\label{Bushnellseq}c(\pi)=n(1+d(\pi)),\end{equation} where~$d(\pi)$ is the \emph{depth} of~$\pi$.
\subsection{Conductors modulo~$\ell$}
The existence of newforms for cuspidal~$\Fl$-representations follows immediately by reduction modulo~$\ell$, and if~$\pi$ is an integral~$\Ql$-representation of~$\G_n$ we have~$c(r_{\ell}(\pi))\leqslant c(\pi)$. % In the supercuspidal case, it a simple matter to show equality. 

\begin{remark}\label{rkconductorsliftsequal}
Let~$\widetilde{\pi},\widetilde{\pi}'$ be integral cuspidal~$\Ql$-representations of~$\G_n$ with~$n\geqslant 2$.  Suppose~$r_{\ell}(\widetilde{\pi})=r_{\ell}(\widetilde{\pi}')$, then~$\widetilde{\pi}$ and~$\widetilde{\pi}'$ have the same depth, and hence the same conductor by Bushnell's equation \eqref{Bushnellseq}.  
\end{remark} 

\begin{proposition}\label{Scconductor}
Let~$\pi$ be a supercuspidal~$\Fl$-representation of~$\G_n$ with $n\geqslant 2$, and~$\widetilde{\pi}$ a supercuspidal lift of~$\pi$.  Then
$c(\pi)=
c(\widetilde{\pi})$.
\end{proposition}

This essentially follows from adapting the argument of Cui--Lanard--Lu \cite[Theorem 3.4]{CLL} (who show one can lift distinction by a closed subgroup in this setting -- we apply similar ideas to lifting invariants by~$\K_n(m)$).  In fact, the approach of ibid.~simplifies in our case as~$\K_n(m)$ is a compact open subgroup of~$\G_n$, and we give the full details:
\begin{proof}
%The case~$n=1$ is easy so we assume~$n\geqslant 2$.   
We have already seen in Remark \ref{rkconductorsliftsequal} that~$c(\widetilde{\pi})$ is independent of the choice of lift~$\widetilde{\pi}$, and~$c(r_{\ell}(\pi))\leqslant c(\pi)$.  % in this case.
It thus suffices to construct a lift~$\widetilde{\pi}$ with a vector invariant by~$\K_n(c(\pi))$.  We can choose a model from type theory for~$\pi$ :
\[\pi\simeq \ind_{\J \E^\times}^{\G_n}(\Lambda),\] 
with \emph{Bushnell--Kutzko--Vign\'eras supercuspidal type}~$(\J\E^\times,\Lambda)$, where~$\J$ is the maximal compact open subgroup of~$\J\E^\times$ and~$\E$ is a field extension of~$\F$ embedded in~$\mathrm{M}_n(\F)$.  As in \cite{BK93}, summarised and extended to the modular setting in \cite[III]{Vig96}, the compact open subgroup~$\J$ is constructed together with a pro-$p$ normal subgroup~$\J^1$ of~$\J$ such that :
\begin{enumerate}
\item $\Lambda_{|\J}=\kappa\otimes \tau$ where~$\kappa$ is a ``beta extension'' of an irreducible ``Heisenberg representation''~$\eta$ of~$\J^1$;
\item $\tau$ is trivial on~$\J^1$ and identifies with an irreducible supercuspidal representation of~$\GL_{s}(k_{\E})\simeq \J/\J^1$ where~$s=s(\pi)$ is an invariant of~$\pi$ called its \emph{relative degree}.
\end{enumerate}
Let~$m=c(\pi)$,~i.e.~$\pi^{\K_n(m)}\neq 0$.  Then by Mackey Theory we have :
\begin{align*}
0\neq \Hom_{\Fl[\K_n(m)]}(1,\ind_{\J \E^\times}^{\G_n}(\Lambda))%&\simeq \Hom_{\Fl[\K_n(m)]}(1,\Ind_{\J \E^\times}^{\G_n}(\Lambda))\\
&\simeq \Hom_{ \Fl[\K_n(m)]}(1,\bigoplus_{}\ind_{(\J \E^\times)^g\cap \K_n(m)}^{\K_n(m)}(\Lambda^g))\\
&\simeq  \Hom_{ \Fl[\K_n(m)]}(1,\bigoplus_{}\ind_{\J^g\cap \K_n(m)}^{\K_n(m)}(\Lambda^g)),
%&\simeq \prod_{}\Hom_{\Fl[(\J \E^\times)^g\cap \K_n(m)]}(1,\Lambda^g),\\
%&\simeq \bigoplus_{}\Hom_{\Fl[\J ^g\cap \K_n(m)]}(1,\kappa^g\otimes\tau^g),
\end{align*}
as~$\K_n(m)$ is compact hence~$(\J\E^\times)^g\cap \K_n(m)=\J^g\cap \K_n(m)$.  
Hence, there exists~$g\in\G$, such that
\[0\neq  \Hom_{ \Fl[\K_n(m)]}(1,\ind_{\J^g\cap \K_n(m)}^{\K_n(m)}(\Lambda^g))\simeq \Hom_{\Fl[\J^g\cap \K_n(m)]}(1,\kappa^g\otimes\tau^g).\]
By conjugating the type which does not change the isomorphism class of~$\pi$, we can assume that the trivial double coset contributes a non-zero Hom-space.  Now
\[0\neq \Hom_{\Fl[\J \cap \K_n(m)]}(1,\kappa\otimes\tau)\hookrightarrow \Hom_{\Fl[\J \cap \K_n(m)]}(1,\kappa\otimes (P_{\tau}\otimes\Fl)),\]
where~$P_{\tau}$ is the projective~$\Zl[\GL_s(k_\E)]$-envelope of~$\tau$.  
Moreover, as~$\tau$ is supercuspidal,
\begin{enumerate}
\item $P_{\tau}\otimes \Fl$ is in fact~$\tau$-isotypic;
\item $P_{\tau}$ is cuspidal, with~$P_{\tau}\otimes \Ql\simeq \bigoplus\widetilde{\tau}$ where the sum is over all supercuspidal~$\Ql$-representations~$\widetilde{\tau}$ of~$\GL_s(k_\E)$ with $r_{\ell}(\widetilde{\tau})=\tau$.%_{\substack{\widetilde{\tau}\text{ supercuspidal}\\r_{\ell}(\widetilde{\tau})=\tau}}
\end{enumerate}
We lift~$\kappa$ to an integral beta extension~$\kappa_{\Zl}$, then the representation~$\kappa_{\Zl}\otimes P_{\tau}$ is a projective envelope of~$\kappa_{\Fl}\otimes \tau$ in the category of~$\Zl[\J]$-modules -- cf.~\cite[Lemma 4.8]{HelmBC}.   
%\begin{lemma}
Moreover, if~$\R=\Ql$ or $\R=\Fl$ we have
\begin{align*}
\Hom_{\Zl[\J \cap \K_n(m)]}(1_{\Zl},\kappa_{\Zl}\otimes P_{\tau})\otimes \R&\simeq \Hom_{\R[\J \cap \K_n(m)]}(1_{\R},(\kappa\otimes P_{\tau})\otimes\R),
\end{align*}
and $\Hom_{\Zl[\J \cap \K_n(m)]}(1_{\Zl},\kappa_{\Zl}\otimes P_{\tau})$ is~$\ell$-torsion free (as~$\kappa_{\Zl}\otimes P_{\tau}$ is~$\ell$-torsion free).

Thus, as~$\Hom_{\Zl[\J \cap \K_n(m)]}(1,\kappa_{\Zl}\otimes P_{\tau})\otimes \Fl\neq 0$, we have
\[0\neq \Hom_{\Zl[\J \cap \K_n(m)]}(1,\kappa_{\Zl}\otimes P_{\tau})\otimes \Ql\simeq\bigoplus_{\widetilde{\tau}: r_{\ell}(\widetilde{\tau})=\tau}\Hom_{\Ql[\J \cap \K_n(m)]}(1,\kappa_{\Ql}\otimes \widetilde{\tau}).\]
Hence there exists~$\widetilde{\tau}$ such that~$\Hom_{\Ql[\J \cap \K_n(m)]}(1,\kappa_{\Ql}\otimes \widetilde{\tau})\neq 0$, and reversing the Mackey theory over~$\Ql$ for~$\widetilde{\pi}=\ind_{\J\E^\times}^{\G_n}(\widetilde{\Lambda})$, for any extension~$\widetilde{\Lambda}$ of~$\kappa_{\Ql}\otimes \widetilde{\tau}$ to~$\J\E^\times$ (in particular for a lift~$\widetilde{\pi}$ of~$\pi$), we find that~$c(\pi)=m\geqslant c(\widetilde{\pi})$, hence~$c(\pi)=c(\widetilde{\pi})$.
\end{proof}

\begin{remark}\label{cuspnonscex}
We expect that the proposition holds for the more general case of \emph{cuspidal~$\Fl$-representations} of~$\G_n$ with~$n\geqslant 2$, and will prove this in the special case of depth zero cuspidal~$\Fl$-representations.  However, for a cuspidal non-supercuspidal~$\Fl$-representation there exist irreducible non-supercuspidal integral~$\Ql$-representations which contain the cuspidal representation as a subquotient on reduction modulo~$\ell$, and the conductors of these can differ as the following example shows:  Let~$\G_2=\GL_2(\F)$ and for a character~$\chi$ of~$\F^\times$ denote by~$\St_2(\chi)$ the \emph{generalized Steinberg}~$\Ql$-representation which is the unique irreducible quotient of~$\chi \nu^{1/2}\times\chi\nu^{-1/2}$.    We have
\begin{equation*}
c(\St_2(\chi))=\begin{cases} 1&\text{if~$\chi$ is unramified,} \\
2&\text{if~$\chi$ is ramified.}\end{cases}\end{equation*}  Suppose~$\chi$ is integral unramified and~$\ell\mid q+1$.  Then~$r_{\ell}(\St_2(\chi))$ contains a depth zero cuspidal non-supercuspidal subquotient~$\pi$ (see for example \cite{VignerasComp}) which has conductor~$2$ by Theorem \ref{depthzerocoset}.  However, for such an unramified~$\chi$,~$1=c(\St_2(\chi))\neq c(\pi)=2.$  

Notice also the analogue of Proposition \ref{Scconductor} for characters of~$\G_1$ only holds for ramified characters.   This leads us to expect that subtleties occur whenever the supercuspidal support of a generic representation contains an unramified character.
\end{remark}

\subsection{Newforms vectors in models}
With applications to test vector problems in mind, one can ask for an explicit description of the newform in an explicit model of a representation.

\subsubsection{Whittaker models}
\begin{thm}[{\cite{Matringe}}, {\cite{Miyauchi}}]\label{MatMiy}
Let~$\pi$ be a generic $\mathbb{C}$-representation of~$\G_n$ with~$n\geqslant 2$, and~$\psi:\F\rightarrow \mathbb{C}^\times$ of conductor~$0$.  There exists a unique Whittaker function~$W_{\new,\pi}\in W(\pi,\psi)$, which is right~$\K_{n-1}$-invariant, and which satisfies, for~$
t=\diag(t_1,\ldots,t_{n-1}) \in \T_{n-1}$ the equality :
\begin{align*}
W_{\new,\pi}(\diag(t,1))=
\begin{cases}
W_{0,\pi_u}(t_r)\nu(t_r)^{\frac{n-r}{2}}\prod_{r<i<n} \mathbf{1}_{\mathfrak{o}_{\F}^\times}(t_i)&\text{when~}r\geqslant 1;\\
\prod_{1<i<n} \mathbf{1}_{\mathfrak{o}_{\F}^\times}(t_i)&\text{when~}r=0;\end{cases}
\end{align*}
where~$t_r=\diag(t_1,\ldots,t_{r})\in\T_{r}$, and~$\pi_u$ is an unramified standard module of~$\G_r$ associated to~$\pi$ in \cite[Definition 1.3]{Matringe}.  
\end{thm}
This formula characterizes the local newform (normalized at~$1$) in the Whittaker model.  Note that, if~$\pi$ is cuspidal (and not an unramified character of~$\G_1$), then~$\pi_u$ is trivial.  

For a generic~$\Ql$-representation~$\pi$, the map~$\Res_{\P_n}:W\rightarrow W_{|_{\P_n}}$ on~$W(\pi,\psi_{\A})$ is injective, and we denote its image (which gives the \emph{Kirillov} model for~$\pi$) by~$\mathcal{K}(\pi,\psi)$.   By the Iwasawa decomposition~$\P_n=\U_n \T_{n-1}\K_{n-1}$, the Matringe--Miyauchi formula gives the description of~$\Res_{\P_n}(\W_{\new,\pi})$, i.e.~describes its image in the Kirillov model~$\mathcal{K}(\pi,\psi)$.  Hence, by right~$\K_n(c(\pi))$-invariance, it describes~$\W_{\new,\pi}$ on~$\P_n\K_n(c(\pi))$.

\subsubsection{Matrix coefficients in the cuspidal case}
Let~$(\pi,\mathcal{V})$ be a cuspidal~$\R$-representation of~$\G_n$, with contragredient~$(\pi^\vee,\mathcal{V}^{\vee})$, and let~$c(\pi),c(\pi^{\vee})$ denote their conductors.   %Let~$v_{\new}\in\mathcal{V}^{\K_n(c(\pi))}$ and
%Let~$v_{\new}\in(\mathcal{V})^{\K_n(c(\pi))}, 
Let~$v^{\vee}_{\new}\in(\mathcal{V}^{\vee})^{\K_n(c(\pi))}$ be non-zero, aka a newform in~$\pi^{\vee}$.  We have the standard intertwiner
\begin{align*}
\pi&\rightarrow \ind_{\Z}^{\G}(\omega_\pi)\\
v&\mapsto c_{v,v^{\vee}_{\new}}:g\mapsto \langle \pi(g) v,v^{\vee}_{\new}\rangle.
\end{align*}

\begin{lemma}
Let~$\pi$ be a cuspidal~$\R$-representation such that~$c(\pi)=c(\pi^{\vee})$, and~$\pi^{\K_n(c(\pi))}\simeq (\pi^{\vee})^{\K_n(c(\pi))}\simeq\R$ (known for example, if~$\R=\Ql$ by Theorem \ref{JPSSMatringe}).
\begin{enumerate}
\item If~$v\in\mathcal{V}^{\K_n(c(\pi))}$ is non-zero then the coefficient~$c_{v,v^{\vee}_{\new}}$ is non-zero.
\item Set~$c_{\pi,\new}=c_{v_{\new},v^{\vee}_{\new}}$.  For~$k,k'\in \K_n(c(\pi))$, and~$g\in\G$, we have
\begin{equation*}%\label{coeff2}\tag{coeff2}
c_{\pi,\new}(kgk')=c_{\pi,\new}(g),
\end{equation*}
and $c_{\pi,\new}$ is the unique matrix coefficient of~$\pi$ satisfying this up to scalar.
\end{enumerate}
\end{lemma}

\begin{proof}
Note that~$\omega_{\pi}\mid_{ \K_n(c(\pi))\cap \Z}=1$, and the map
\[\pi\otimes_{\R}\pi^{\vee}\rightarrow \ind_{\Z}^{\G}(\omega_{\pi}),\]
is an injective~$(\G\times \G)$-module morphism for the action~$(g,g')\cdot f(x)=f(g'^{-1} xg)$ on~$\ind_{\Z}^{\G}(\omega_{\pi})$.   By left exactness of~$\K_n(c(\pi))\times\K_n(c(\pi))$-invariants, this induces an injective morphism
\[\R=\pi^{\K_n(c(\pi))}\otimes(\pi^{\vee})^{\K_n(c(\pi))}=(\pi\otimes\pi^{\vee})^{\K_n(n)\times\K_n(n)}\hookrightarrow (\ind_{\Z}^{\G}(\omega_\pi))^{\K_n(c(\pi))\times\K_n(c(\pi))}.\]
This establishes the first statement and the second follow from this.
\end{proof}

\subsection{$\varepsilon$-factor conductor}
Let~$\psi:\F\rightarrow \mathbb{C}^\times$ be a non-trivial character.  Let~$\pi$ be an irreducible~$\mathbb{C}$-representation of~$\G_n$.  Then we have a local Godement--Jacquet epsilon factor associated to~$\pi$ : $\varepsilon(s,\pi,\psi)$ which is a monomial in~$q^{-s}$, and we can write
\[\varepsilon(s,\pi,\psi)=\varepsilon(0,\pi,\psi)q^{-c_{\varepsilon,\psi}(\pi)s}.\]
for a non-negative integer~$c_{\varepsilon,\psi}(\pi)$ called the \emph{Godement--Jacquet conductor}.

\begin{proposition}[{\cite{JPSS}}]\label{JPSSepsilon}
Let~$\pi$ be an irreducible generic~$\mathbb{C}$-representation of~$\G_n$.  Then
\[c(\pi)=c_{\varepsilon,\psi}(\pi)-nc(\psi).\]
\end{proposition}

Replacing~$q^{-s}$ with an indeterminate~$\mathrm{X}$ and working with~$\R$-valued Haar measures, M\'inguez extends the Godement--Jacquet construction to~$\R$-representations in~\cite{Minguez} (and the Rankin--Selberg construction is extended to~$\R$-representations in \cite{KM}).  The~epsilon factor in this setting is a monomial in~$X$.  And with this change of variable, for~$\pi$ a generic~$\R$-representation of~$\G_n$, for a non-trivial character~$\psi:\F\rightarrow \R^\times$, we can define~$c_{\varepsilon,\psi}(\pi)\in\mathbb{Z}$ by
\[\varepsilon(X,\pi,\psi)=a X^{c_{\varepsilon,\psi}(\pi)},\]
for some~$a\in\R^\times$.   

For cuspidal representations, the epsilon factor agrees with the gamma factor which is compatible with congruences mod~$\ell$ where we first fix a non-trivial character~$\psi:\F\rightarrow \Zl^\times$  (cf.~\cite{Minguez,KM}), hence we obtain from Propositions \ref{Scconductor} and \ref{JPSSepsilon}:

\begin{lemma}
Let~$\psi:\F\rightarrow \Zl^\times$ be a non-trivial character.
\begin{enumerate}
\item Let~$\widetilde{\pi}$ be an integral~$\Ql$-cuspidal of~$\G_n$, then~$c_{\varepsilon,\psi\otimes \Ql}(\pi)=c_{\varepsilon,\psi\otimes \Fl}(r_{\ell}(\widetilde{\pi}))$.  
\item Let~$\pi$ be a supercuspidal~$\Fl$-representation of~$\G_n$.  Then~$c(\pi)=c_{\varepsilon,\psi\otimes\Fl}(\pi)-nc(\psi\otimes \Fl).$
\end{enumerate}
\end{lemma}

\section{Newform vectors in depth zero cuspidal representations}\label{Section:depthzero}

Let~$\pi$ be a depth zero cuspidal~$\R$-representation of~$\G_n$. Choose a Moy--Prasad--Morris--Vign\'eras model \cite[III 3.3]{Vig96} for~$\pi$ as compactly induced
\[\ind_{\Z_n\K_n}^{\G_n}(\widetilde{\tau}),\]
where~$\widetilde{\tau}=\omega_{\pi}\tau$ and~$\tau$ is a cuspidal representation of~$\GL_n(k_{\F})$ which we identify with its Whittaker model~$W(\tau,\overline{\psi})$ with respect to~$(\U_n(k_\F),\overline{\psi})$ where~$\U_n$ denotes the standard upper triangular unipotent of~$\GL_n$. 
We define the \emph{Bessel vector}~$\mathcal{B}_{\widetilde{\tau}}\in\ind_{\Z_n\K_n}^{\G_n}(\widetilde{\tau})$ by $\mathrm{supp}(\mathcal{B}_{\widetilde{\tau}})\subseteq\Z_n\K_n$, and
\[\mathcal{B}_{\widetilde{\tau}}(zk) = \omega(z)\, \overline{k}\cdot B_{\tau,\overline{\psi}},\]
for all~$z\in\Z_n, k\in\K_n$.  Notice that for all~$u\in(\K_n\cap \U_n)\K_n^1$ we have~$\mathcal{B}_{\widetilde{\tau}}(-u)=\overline{\psi}(\overline{u})\mathcal{B}_{\widetilde{\tau}}(-)$, and~$\mathcal{B}_{\widetilde{\tau}}$ is the unique up to scalar vector in~$((\K_n\cap \U_n)\K_n^1,\overline{\psi})$-isotypic space (cf.~\cite[Corollary 4.5]{PaskSte}).

For~$n>2$, we write~$\B_{n-1}^{\op}(k_\F)$ for the Borel subgroup of lower triangular matrices in~$\GL_{n-1}(k_\F)$ which we consider as a subgroup of~$\GL_{n}(k_\F)$ via the standard embedding~$g\mapsto \left(\begin{smallmatrix} g&0\\0&1\end{smallmatrix}\right)$ of~$\GL_{n-1}(k_\F)$ in~$\GL_{n}(k_\F)$.  We set~$\B_{1}^{\op}(k_\F)=\GL_1(k_\F)$, which we also consider as a subgroup of~$\GL_2(k_\F)$ by the standard embedding.

Let
\[\Sigma_n=\left(\begin{smallmatrix}\varpi_\F^{n-1}&&&\\&\varpi_\F^{n-2}&&\\&&\ddots&\\&&&1\end{smallmatrix}\right).\]

The aim of this section is to prove:
\begin{thm}
\label{depthzerocoset}
Let~$n\geqslant 2$, and~$\pi=\ind_{\Z_n\K_n}^{\G_n}(\widetilde{\tau})$ be a depth zero cuspidal~$\R$-representation with cuspidal~$\R$-type~$(\K_n,\tau)$.
\begin{enumerate}
\item\label{depthzeromain1} (newforms) Then~$c(\pi)=n$,~$\Hom_{\R[\K_n(c(\pi))]}(1,\pi)\simeq \R$, and the unique up to scalar (non-zero newform)~$f_{\new}\in\pi^{\K_n(n)}$ has support~$\mathrm{supp}(f_{\new})\subseteq \Z_n\K_n\Sigma_n\K_n(n)$.
\item \label{depthzeromain2} (Reeder's oldforms) for~$m>c(\pi)=n$
\[\dim_{\R}(\Hom_{\R[\K_n(m)]}(1,\pi))=\begin{pmatrix}m-1\\n-1 \end{pmatrix}.\]
\item \label{depthzeromain3} (explicit formula in terms of Bessel functions) The function~$f_{\new}\in \pi^{\K_n(n)}$ is characterized by
\[ f_{\new}(\Sigma_n)=\sum_{b\in \B_{n-1}^{\op}(k_\F)}b\cdot B_{\tau,\overline{\psi}}.\]
Moreover, there exists a unique~$\R$-Haar measure on~$\K_n(n)$ such that, for all~$g\in\G$, 
\[f_{\new}(g)=\int_{\K_n(n)} \mathcal{B}_{\widetilde{\tau}}(gk\Sigma_n^{-1}) dk.\]
\end{enumerate}
\end{thm}

\begin{remark}
In the special case,~$\R=\mathbb{C}$, parts \eqref{depthzeromain1} and \eqref{depthzeromain2} follows from Reeder in \cite[(2.3) Example]{Reeder}  which uses: the existence and uniqueness of newforms \cite{JPSS,Jacquet, Matringe}, and the prior knowled ~$c(\pi)=n$ \cite{Bushnell}.  For~$\GL_2(\F)$ this is made further explicit \cite{Knightly} giving an explicit formula.   In this section we will give a self-contained proof (just using the basic tools of Mackey theory) of this theorem, which applies equally well to the broader case including~$\ell$-modular representations where~$\R$ has positive characteristic~$\ell\neq p$.
\end{remark}

\subsection{Mackey theory}\label{secmackey}
For a positive integer~$m$, by Mackey theory we have :
\begin{align*}
0\neq \Hom_{\R[\K_n(m)]}(1,\ind_{\Z_n\K_n}^{\G_n}(\omega_\pi\tau))
&\simeq \Hom_{ \R[\K_n(m)]}(1,\bigoplus_{\Z_n\K_n\backslash \G_n/\K_n(m)}\ind_{(\Z_n\K_n)^g\cap \K_n(m)}^{\K_n(m)}(\omega_\pi^g\tau^g))\\
&\simeq  \Hom_{ \R[\K_n(m)]}(1,\bigoplus_{\Z_n\K_n\backslash \G_n/\K_n(m)}\ind_{\K_n^g\cap \K_n(m)}^{\K_n(m)}(\tau^g)),
\end{align*}
as~$\K_n(m)$ is compact hence~$(\Z_n\K_n)^g\cap \K_n(m)=\K_n\cap \K_n(m)$.   We are thus reduced to studying the spaces
\[  \Hom_{ \R[\K_n(m)]}(1,\ind_{\K_n^g\cap \K_n(m)}^{\K_n(m)}(\tau^g))\simeq \Hom_{\R[\K_n^g\cap \K_n(m)]}(1,\tau^g),\]
over a set of double coset representatives for~$\Z_n\K_n\backslash \G_n/\K_n(m)$.  Equivalently we can study the spaces
\[ \Hom_{\R[\K_n\cap \presuper{g}\K_n(m)]}(1,\tau)\simeq \Hom_{\R[\overline{\K_n\cap \presuper{g}\K_n(m)}]}(1,\tau)\]
where~$\overline{\K_n\cap \presuper{g}\K_n(m)}$ denotes the image of~$\K_n\cap \presuper{g}\K_n(m)$ in~$\K_n/\K_n^1$.
%	where $\tau$ is a cuspidal representation of $\operatorname{GL}_n(k_\F)$.
	
\subsection{Coset representatives}
	
	 For positive integers $i$ and $j$ and a commutative ring~$\S$, let 
	 \[\N_{i,j}(\S)=\left\{\left(\begin{smallmatrix}
			1_i&\\
			X&1_j
		\end{smallmatrix}\right)\in\M_{(i+j)\times (i+j)}(\S): X\in \M_{j\times i}(\S)\right\}.\]
	   be the unipotent radical of the lower triangular standard parabolic with block sizes $i$ and $j$.
	   
	   We introduce the following notation:
\begin{enumerate}[-]	   
	\item	If $v$ is an element of $\M_{i\times j}(\mathfrak o_{\F}/\mathfrak p_{\F}^m)$ we write $\tilde{v}$ for a lift of $v$ in $\M_{i\times j}(\mathfrak o_{\F})$. If $v=0$ we choose $\tilde{v}=0$. 
        \item For any matrix $X\in \M_{i\times j}(\mathfrak o_{\F})$ we write $\overline{X}\in \M_{i\times j}(k_{\F})$ for its reduction modulo $\mathfrak p_{\F}$. 
        \item For any subset $S=\{s_i\}_{i\in I}\subseteq \M_{i\times j}(\mathfrak o_{\F})$ we will write $\overline{S}$ for $\{\overline{s_i}\}_{i\in I}$.
        \end{enumerate}
		
		\begin{lemma}\label{cosets1} For $n\geqslant 2$ and $m\geqslant 1$ the union of 
		\[\mathcal A^1_n=\left\{\left(\begin{smallmatrix}
				1_{n-1}&0\\
				\tilde{v}&\tilde{x}
			\end{smallmatrix}\right)\mid v\in \M_{1\times (n-1)}(\mathfrak o_{\F}/\mathfrak p_F^m),~x\in (\mathfrak o_{\F}/\mathfrak p_F^m)^\times\right\}\]
			and the set
			\[\mathcal A^2_n=\left\{\left(\begin{smallmatrix}
				1_{n-j-1}&&&\\
				\tilde{w}_1&\tilde{w}_2&\tilde{w}_3&\tilde{x}\\
				&&1_{j-1}&\\
				&1&&
			\end{smallmatrix}\right)
			\mid \begin{array}{cc}
				&w_1\in \M_{1\times (n-j-1)}(\mathfrak o_{\F}/\mathfrak p_F^m),~w_2\in\mathfrak p_F/\mathfrak p_F^m,\\
				&w_3\in \M_{1\times (j-1)}(\mathfrak p_F/\mathfrak p_F^m),~x\in(\mathfrak o_{\F}/\mathfrak p_F^m)^\times
			\end{array}\right\}\]
			is a set of coset representatives for $\K_n/\K_n(m)$.
		\end{lemma}
		\begin{proof} Note that $\K_n=\GL_n(\mathfrak o_{\F})$ acts on $\M_{1\times n}(\mathfrak o_{\F}/\mathfrak p_F^m)$, i.e.\ row vectors with entries in $\mathfrak o_{\F}/\mathfrak p_F^m$, via reduction modulo $\mathfrak p_F^m$ and multiplication from the right. Then $\K_n(m)$ is exactly the stabilizer of $(0,\dotsc,0,1)$. We show that a vector $v=(v_{n-1},\dotsc,v_0)\in \M_{1\times n}(\mathfrak o_{\F}/\mathfrak p_F^m)$ lies in the orbit of $(0,\dotsc,0,1)$ if and only if there is a $0\leqslant j\leqslant n-1$ such that $v_j\in (\mathfrak o_{\F}/\mathfrak p_F^m)^\times$. It is clear that if all the $v_i$ are elements of $\mathfrak p_F/\mathfrak p_F^m$ there can be no element in $\K_n$ that maps $(0,\dotsc,0,1)$ to $v$.\
		
		Suppose that $v_0\in (\mathfrak o_{\F}/\mathfrak p_F^m)^\times$, then 
			\[\left(\begin{smallmatrix}
				1_{n-1}&0\\
				\tilde{v}_{n-1}\dotsb\tilde{v}_{1}&\tilde{v}_0
			\end{smallmatrix}\right)\]
			is an element of $\K_n$ that maps $(0,\dotsc,0,1)$ to $v$.
			
			If $v_0\in \mathfrak p_F/\mathfrak p_F^m$ let $j$ be the smallest positive integer such that $v_j\in (\mathfrak o_{\F}/\mathfrak p_F^m)^\times$, then
			\[\left(\begin{smallmatrix}1_{n-j-1}&&&\\
				&0&&1\\
				&&1_{j-1}\\
				\tilde{v}_{n-1}\dotsb\tilde{v}_{j+1}&\tilde{v}_j&\tilde{v}_{j-1}\dotsb\tilde{v}_{1}&\tilde{v}_0
			\end{smallmatrix}\right)\]
			is an element of $\K_n$ that maps $(0,\dotsc,0,1)$ to $v$. By taking inverses we obtain the result.
		\end{proof}
		Note that by the Cartan decomposition the set 
		$$\mathcal B_n=\left\{\left(\begin{smallmatrix}
			\varpi_\F^{\alpha_{n-1}}&&&&\\
			&\varpi_\F^{\alpha_{n-2}}&&&\\
			&&\ddots&&\\
			&&&\varpi_\F^{\alpha_1}&\\
			&&&&1
		\end{smallmatrix}\right)\mid\alpha_i\in\mathbb Z,~0\leqslant\alpha_1\leqslant\alpha_2\leqslant\dotsc\leqslant\alpha_{n-1}\right\}$$
		is a set of double coset representatives for $\K_n\Z_n\backslash \G_n/\K_n$. This together with Lemma \ref{cosets1} gives us an exhaustive list of representatives for $\K_n\Z_n\backslash \G_n/\K_n(m)$. We will need to following slight refinement.
		\begin{proposition}
			The union of 
			$$\mathcal C_n=\left\{\left(\begin{smallmatrix}
				\varpi_{\F}^{\alpha_{n-1}}&&&&\\
				&\varpi_{\F}^{\alpha_{n-2}}&&&\\
				&&\ddots&&\\
				&&&\varpi_{\F}^{\alpha_{1}}&\\
				\tilde{v}_{n-1}&\tilde{v}_{n-2}&\dotsc&\tilde{v}_{1}&1
			\end{smallmatrix}\right)
			\mid
			\begin{array}{cc}
				&0\leqslant\alpha_1\leqslant\alpha_2\leqslant\dotsc\leqslant\alpha_{n-1}\\
				&v_i\in \M_n(\mathfrak o_{\F}/\mathfrak p_F^m)\\
				&\operatorname{val}_{\F}(\tilde{v}_i)<\alpha_i
			\end{array}\right\}$$
			and $\mathcal A^2_n\cdot\mathcal B_n$ is an exhaustive collection of coset representatives for $\K_n\Z_n\backslash \G_n/\K_n(m)$. 
		\end{proposition}
		\begin{proof}
			As already mentioned by the above and Lemma \ref{cosets1} we know that $\mathcal A^1_n\cdot\mathcal B_n\cup\mathcal A^2_n\cdot\mathcal B_n$ yields an exhaustive collection of coset representatives for $\K_n\Z_n\backslash \G_n/\K_n(m)$. Let $X$ be any element in $\mathcal A^1_n\cdot\mathcal B_n$. Then
			$$X=\left(\begin{smallmatrix}
				A&0\\
				\tilde{v}&\tilde{x}
			\end{smallmatrix}\right)$$ where $A=\operatorname{diag}(\varpi_{\F}^{\alpha_{n-1}},\dotsc,\varpi_{\F}^{\alpha_{1}})$, $\tilde{x}\in\mathfrak o_{\F}^\times$ and $\tilde{v}=(\tilde{v}_{n-1},\dotsc,\tilde{v}_{1})\in \M_{1\times(n-1)}(\mathfrak o_{\F})$.
			Suppose there is $1\leqslant j\leqslant n-1$ such that $\operatorname{val}_{\F}(\tilde{v_j})\geqslant\alpha_j$. Then the matrix
			$$\left(\begin{smallmatrix}
				1_{n-j-1}&&&\\
				&1&&\\
				&&1_{j-1}&\\
				&-\tilde{v}_j\varpi_{\F}^{-\alpha_j}&&1
			\end{smallmatrix}\right)$$
			lies in $\K_n$ and
			$$\left(\begin{smallmatrix}
				1_{n-j-1}&&&\\
				&1&&\\
				&&1_{j-1}&\\
				&-\tilde{v}_j\varpi_{\F}^{-\alpha_j}&&1
			\end{smallmatrix}\right)X=\left(\begin{smallmatrix}
				&A&&0\\
				\tilde{v}_{n-1}\dotsc\tilde{v}_{j+1}&0&\tilde{v}_{j-1}\dotsc\tilde{v}_{1}&\tilde{x}
			\end{smallmatrix}\right)$$
			generates the same double coset as $X$. Hence we can assume that $\operatorname{val}_{\F}(\tilde{v}_i)<\alpha_i$ for all $1\leqslant i\leqslant n-1$. Moreover, by multiplying $X$ by the left by the matrix $\operatorname{diag}(1,\dotsc,1,\tilde{x}^{-1})\in \K_n$ we see that 
			$$\left(\begin{smallmatrix}
				A&0\\
				\tilde{x}^{-1}\tilde{v}&1
			\end{smallmatrix}\right)$$ lies in the same double coset as $X$, which implies the result.
		\end{proof}

		\subsection{Proof of Theorem \ref{depthzerocoset}}		
		
		We now consider the Hom-space~$ \Hom_{\R[\overline{\K_n\cap \presuper{g}\K_n(m)}]}(1,\tau)$ over various cases of our chosen coset representatives for~$\K_n\Z_n\backslash \G_n/\K_n(m)$.  Essentially, the cuspidality of~$\tau$ will force most of these spaces to be zero.
		
		\begin{proposition}
			For any $n\geqslant 2$ and $m\geqslant 1$ if $g\in\mathcal A^2_n\cdot\mathcal B_n$, then $ \Hom_{\R[\overline{\K_n\cap \presuper{g}\K_n(m)}]}(1,\tau)=0$.
		\end{proposition}
		\begin{proof}
			Let $g\in\mathcal A^2_n\cdot\mathcal B_n$, which then has the form
			$$g=\left(\begin{smallmatrix}
				A&&&\\
				\varpi_{\F}^{\alpha_j}w_1&\varpi_{\F}^{\alpha_j}w_2&\varpi_{\F}^{\alpha_j}w_3&\varpi_{\F}^{\alpha_j}x\\
				&&B&\\
				&1&&
			\end{smallmatrix}\right),$$
			where  $w_1\in \M_{1\times (n-j-1)}(\mathfrak o_{\F}),w_2\in\mathfrak p_F,w_3\in \M_{1\times (j-1)}(\mathfrak p_F),x\in\mathfrak o_{\F}^\times$, $A=\operatorname{diag}(\varpi_{\F}^{\alpha_{n-1}},\dotsc,\varpi_{\F}^{\alpha_{j+1}})$, and $B=\operatorname{diag}(\varpi_{\F}^{\alpha_{j-1}},\dotsc,\varpi_{\F}^{\alpha_{1}})$, for integers $0\leqslant\alpha_1\leqslant\dotsc\leqslant\alpha_{n-1}$. 
			For $a\in \M_{1\times(n-j-1)}(\mathfrak o_{\F}),~b\in \M_{1\times(j-1)}(\mathfrak o_{\F}),~c\in\mathfrak o_{\F}$, let 
			$$X(a,b,c)=\left(\begin{smallmatrix}
				1_{n-j-1}&&&\\
				aA+ \varpi_{\F}^{\alpha_j}cw_1&1&bB+ \varpi_{\F}^{\alpha_j}cw_3&x \varpi_{\F}^{\alpha_j}c\\
				&&1_{j-1}&\\
				&&&1
			\end{smallmatrix}\right)$$ which is an element of $\K_n(m)$. Note that
			$$g^{-1}=\left(\begin{smallmatrix}
				A^{-1}&&&\\
				&&&1\\
				&&B^{-1}&\\
				-x^{-1}w_1A^{-1}&\varpi_{\F}^{-\alpha_j}x^{-1}&-x^{-1}w_3B^{-1}&-x^{-1}w_2
			\end{smallmatrix}\right)$$
			and we can compute
			$$gX(a,b,c)g^{-1}=\left(\begin{smallmatrix}
				1_{n-j-1}&&&\\
				\varpi_{\F}^{\alpha_j}w_2a&1+w_2c \varpi_{\F}^{\alpha_j}&\varpi_{\F}^{\alpha_j}w_2b&-\varpi_{\F}^{2\alpha_j}cw_2^2\\
				&&1_{j-1}&\\
				a&c&b&1-\varpi_{\F}^{\alpha_j}cw_2
			\end{smallmatrix}\right).$$
			This matrix is in $\K_n$ and since by assumption $w_2\in\mathfrak p_F$ we obtain
			$$\overline{gX(a,b,c)g^{-1}}=\left(\begin{smallmatrix}
				1_{n-j-1}&&&\\
				&1&&\\
				&&1_{j-1}&\\
				\overline{a}&\overline{c}&\overline{b}&1
			\end{smallmatrix}\right).$$
			Since $a\in \M_{1\times(n-j-1)}(\mathfrak o_{\F}), ~b\in \M_{1\times(j-1)}(\mathfrak o_{\F}), ~c\in\mathfrak o_{\F}$ were arbitrary we see that $\overline{\presuper{g}\K_n(m)\cap \K_n}$ contains $\N_{n-1,1}(k_{\F})$. Hence
			$$\Hom_{\overline{\presuper{g}\K_n(m)\cap \K_n}}(1,\tau)\subseteq\Hom_{\N_{n-1,1}(k_{\F})}(1,\tau),$$
			however by the cuspidality of $\tau$ the latter space is zero.
		\end{proof}
		
		\begin{proposition}
			Suppose $g\in\mathcal C_n$ is not a diagonal matrix. Then $ \Hom_{\R[\overline{\K_n\cap \presuper{g}\K_n(m)}]}(1,\tau)=0$.
		\end{proposition}
		\begin{proof}
			Note that any such $g$ has the form 
			$$g=\left(\begin{smallmatrix}
				A&&&\\
				&\varpi_{\F}^{\alpha_j}&&\\
				&&B&\\
				v&x&0&1
			\end{smallmatrix}\right),$$
			where $1\leqslant j\leqslant n-1, A=\operatorname{diag}(\varpi_{\F}^{\alpha_{n-1}},\dotsc,\varpi_{\F}^{\alpha_{j+1}}),B=\operatorname{diag}(\varpi_{\F}^{\alpha_{j-1}},\dotsc,\varpi_{\F}^{\alpha_{1}}),v\in \M_{1\times(n-j-1)}(\mathfrak o_{\F})$ and $x\not=0$. Let $\K'_{n,j}$ be the subgroup of $\K_n(m)$, defined by,
			\[\K'_{n,j}=\left\{\left(\begin{smallmatrix}
				X&&\\
				Y&1_{j-1}&\\
				0&0&1
			\end{smallmatrix}\right)\mid X\in\operatorname{GL}_{n-j}(\mathfrak o_{\F}), Y\in \M_{(j-1)\times (n-j)}(\mathfrak o_{\F})\right\}.
			\]

			We claim that $\overline{\presuper{g}\K'_{n,j}\cap \K}$ contains the unipotent radical $\N_{n-j,j}(k_{\F})$.  We will show this inductively. As a base case consider, for any $r\geqslant 2$, a matrix $h\in C_r$ of the form
			\[h=\left(\begin{smallmatrix}
				A&&&\\
				&\varpi_{\F}^{\alpha_j}&&\\
				&&B&\\
				0&x&0&1
			\end{smallmatrix}\right)\]
			where $1\leqslant j\leqslant r-1, A=\operatorname{diag}(\varpi_{\F}^{\alpha_{r-1}},\dotsc,\varpi_{\F}^{\alpha_{j+1}}),B=\operatorname{diag}(\varpi_{\F}^{\alpha_{j-1}},\dotsc,\varpi_{\F}^{\alpha_{1}})$ and $x\not=0$. Note that we have $\operatorname{val}_{\F}(x)<\alpha_j$ and since $x\not=0$ that $\alpha_j>0$. 
			For any $a\in \M_{1\times (r-j-1)}(\mathfrak o_{\F}),~b\in\mathfrak o_{\F},~c\in \M_{(j-1)\times(r-j-1)}(\mathfrak o_{\F})$, and $d \in \M_{1\times(j-1)}(\mathfrak o_{\F})$ let
			\[X(a,b,c,d)=\left(\begin{smallmatrix}
				1_{r-j-1}&&&\\
				x^{-1}aA&1+\varpi_{\F}^{\alpha_j}x^{-1}b&&\\
				B^{-1}cA&B^{-1}d\varpi_{\F}^{\alpha_j}&1_{j-1}&\\
				&&&1
			\end{smallmatrix}\right).\] Since $\alpha_i>\operatorname{val}_{\F}(x)$ for all $i\geqslant j$ we have that $x^{-1}aA\in \M_{(r-j-1)\times(r-j-1)}(\mathfrak o_{\F})$. Moreover, the $\alpha_i$ are non-decreasing which implies that $X(a,b,c,d)$ is an element of $\K'_{r,j}$.
			We can compute
			\[hX(a,b,c,d)h^{-1}=\left(\begin{smallmatrix}
				1_{r-j-1}&&&\\
				\varpi_{\F}^{\alpha_j}x^{-1}a&1+\varpi_{\F}^{\alpha_j}x^{-1}b&&\\
				c&d&1_{j-1}&\\
				a&b&0&1
			\end{smallmatrix}\right),\]
			which yields an element of $\K_n$. Clearly, we have
			\[\overline{hX(a,b,c,d)h^{-1}}=\left(\begin{smallmatrix}
				1_{r-j-1}&&&\\
				0&1&&\\
				\overline{c}&\overline{d}&1_{j-1}&\\
				\overline{a}&\overline{b}&0&1
			\end{smallmatrix}\right),\]
			which finishes the base case.
			
			Let $$h_1=\left(\begin{smallmatrix}
				B&&\\
				&C&\\
				v&&1
			\end{smallmatrix}\right)\in\mathcal C_{r_1},$$
			where $B=\operatorname{diag}(\varpi_{\F}^{\alpha_{r_1-1}},\dotsc,\varpi_{\F}^{\alpha_{j}}), C=\operatorname{diag}(\varpi_{\F}^{\alpha_{j-1}},\dotsc,\varpi_{\F}^{\alpha_1})$ and $v\in \M_{1\times (r_1-j)}(\mathfrak o_{\F})$. We assume that $\overline{h_1\K'_{r_1,j}h_1^{-1}\cap K}$ contains $\N_{r_1-j,j}(k_{\F})$.

			Assume now that 
			$$h_2=\left(\begin{smallmatrix}
				A&&&&\\
				0&\varpi_{\F}^{\alpha_{r_1}}&&&\\
				&&&&\\
				&&&h_1&\\
				0&x&&&
			\end{smallmatrix}\right)\in\mathcal C_{r_2}$$ 
			where $A=\operatorname{diag}(\varpi_{\F}^{\alpha_{r_2-1}},\dotsc,\varpi_{\F}^{\alpha_{r_1+1}})$ and $x\not=0$. Then for any $$T=\left(\begin{smallmatrix}
				X&&\\
				Y&1&\\
				0&0&1
			\end{smallmatrix}\right)\in K'_{r_1,j}$$
			we have that 
			$$h_2\left(\begin{smallmatrix}
				1_{r_2-r_1}&\\
				&T
			\end{smallmatrix}\right)h_2^{-1}=\left(\begin{smallmatrix}
				1_{r_2-r_1}&\\
				&h_1Th_1^{-1}
			\end{smallmatrix}\right).$$ By our assumption this implies that $\overline{h_2K'_{r_2,j}h_2^{-1}\cap K}$ contains
			$$\left(\begin{smallmatrix}
				1_{r_2-r_1}&&\\
				&1_{r_1-j}&\\
				&*&1_j
			\end{smallmatrix}\right).$$
			Now let $$Y(a,b,c,d)=\left(\begin{smallmatrix}
				1_{r_2-r_1-1}&&&&\\
				x^{-1}aA&1+\varpi_{\F}^{\alpha_{r_1}}x^{-1}b&&&\\
				&&1_{r_1-j}&&\\
				C^{-1}cA&C^{-1}d\varpi_{\F}^{\alpha_{r_1}}&&1_{j-1}&\\
				&&&&1
			\end{smallmatrix}\right)$$
			where $a\in \M_{1\times(r_2-r_1-1)}(\mathfrak o_{\F}),b\in\mathfrak o_{\F}, c\in \M_{(j-1)\times(r_2-r_1-1)}(\mathfrak o_{\F})$ and $d\in \M_{(j-1)\times 1}(\mathfrak o_{\F})$. By our assumptions on $x$ and the $\alpha_i$, it is easy to see that $Y(a,b,c,d)\in \K'_{r_2,j}$.
			We can compute
			$$h_2Y(a,b,c,d)h_2^{-1}=\left(\begin{smallmatrix}
				1_{r_2-r_1-1}&&&&\\
				\varpi_{\F}^{\alpha_r}x^{-1}a&1+\varpi_{\F}^{\alpha_r}x^{-1}b&&&\\
				&&1_{r_1-j}&&\\
				c&d&&1_{j-1}&\\
				a&b&&&1
			\end{smallmatrix}\right)$$
			and since $\operatorname{val}_{\F}(x)<\alpha_r$ we have
			$$\overline{h_2Y(a,b,c,d)h_2^{-1}}=\left(\begin{smallmatrix}
				1_{r_2-r_1-1}&&&&\\
				&1&&&\\
				&&1_{r_1-j}&&\\
				\overline{c}&\overline{d}&&1_{j-1}&\\
				\overline{a}&\overline{b}&&&1
			\end{smallmatrix}\right).$$
			Hence we obtain that $\overline{\presuper{h_2}\K'_{r_2,j}\cap \K_n}$ contains $\N_{r_2-j,j}(k_{\F})$. Hence
			\[\Hom_{\R[\overline{\K_n\cap \presuper{g}\K_n(m)}]}(1,\tau)\subseteq\Hom_{\R[\N_{r_2-j,j}(k_{\F})]}(1,\tau),\]
			however by the cuspidality of $\tau$ the latter space is zero.
		\end{proof}
		By the above two propositions the only elements $g$ of $\mathcal C_n$ or $\mathcal A^2_n\cdot\mathcal B_n$ for which the space $\Hom_{\R[\overline{\K_n\cap \presuper{g}\K_n(m)}]}(1,\tau)$ can be nonzero are diagonal matrices in $\mathcal C_n$, i.e., elements of $\mathcal B_n$.
		\begin{proposition}
			Suppose that $g=\operatorname{diag}(\varpi_{\F}^{\alpha_{n-1}},\dotsc,\varpi_{\F}^{\alpha_1},1)$ is an element of $\mathcal B_n$ where $\alpha_{n-1}\geqslant m$, then $ \Hom_{\R[\overline{\K_n\cap \presuper{g}\K_n(m)}]}(1,\tau)=0$.
		\end{proposition}
		\begin{proof}
			Let $g'=\operatorname{diag}(\varpi_{\F}^{\alpha_{n-2}},\dotsc,\varpi_{\F}^{\alpha_1})$. Note that by our assumptions for any $v\in \M_{(n-2)\times 1}(\mathfrak o_{\F})$ and $x\in\mathfrak o_{\F}$ the matrix
			$$
			\begin{pmatrix}
				1&&\\
				g'^{-1}v\varpi_{\F}^{\alpha_{n-1}}&1_{n-2}&\\
				\varpi_{\F}^{\alpha_{n-1}}x&&1
			\end{pmatrix}
			$$
			is an element of $\K_n(m)$. We can compute
			$$g \left(\begin{smallmatrix}
				1&&\\
				g'^{-1}v\varpi_{\F}^{\alpha_{n-1}}&1_{n-2}&\\
				\varpi_{\F}^{\alpha_{n-1}}x&&1
			\end{smallmatrix}\right)g^{-1}=\left(\begin{smallmatrix}
				1&&\\
				v&1_{n-2}&\\
				x&&1
			\end{smallmatrix}\right),$$
			and obtain that $\overline{\K_n\cap \presuper{g}\K_n(m)}$ contains $\N_{1,n-1}(k_{\F})$ which yields the result.
		\end{proof}
		\begin{proposition}\label{onedim} 
			Let $g=\operatorname{diag}(\varpi_{\F}^{\alpha_{n-1}},\dotsc,\varpi_{\F}^{\alpha_1},1)\in \mathcal B_n$ where $\alpha_{n-1}<m$.
			\begin{enumerate}
				\item Suppose the $\alpha_i$ are not strictly increasing, i.e.\ there is $1\leqslant j\leqslant n-1$ such that $\alpha_{j-1}=\alpha_{j}$, then $\Hom_{\overline{g\K_n(m)g^{-1}\cap \K}}(1,\tau)$ is zero.  
				\item If for $1\leq i\leq n-2$ we have that $\alpha_i<\alpha_{i+1}$ then $\Hom_{\overline{g\K_n(m)g^{-1}\cap \K}}(1,\tau)$ is one-dimensional.
			\end{enumerate}
			  
		\end{proposition}
		\begin{proof}
			Note that for any matrix $X=(X_{i,j})\in \K_n(m)$ we have $(gXg^{-1})_{i,j}=\varpi_{\F}^{\alpha_{n-i}-\alpha_{n-j}}X_{i,j}$. This implies that $\overline{g\K_n(m)g^{-1}\cap \K}$ is contained in $\P_n$, i.e. the mirabolic subgroup. Since $\tau$ has a Kirillov model we see that
			$$\Hom_{\overline{g\K_n(m)g^{-1}\cap \K}}(1,\tau)=\Hom_{\overline{g\K_n(m)g^{-1}\cap \K}}(1,\operatorname{ind}_{\U_n}^{\P_n}(\psi_n)).$$
			Moreover, $\overline{g\K_n(m)g^{-1}\cap \K}$ contains 
			$$\begin{pmatrix}
				\B_{n-1}^{\op}(k_{\F})&\\
				&1
			\end{pmatrix}$$
			where $\B_{n-1}^{\op}(k_{\F})$ are all invertible lower triangular matrices in $\operatorname{GL}_{n-1}(k_{\F})$.
			
			Let $\Psi$ be an element of $\Hom_{\overline{g\K_n(m)g^{-1}\cap \K}}(1,\operatorname{ind}_{\U_n}^{\P_n}(\psi_n))$ and set $f=\Psi(1)$. Then for any $u\in \U_n,g\in \P_n,b\in \B_{n-1}^{\op}(k_{\F})$ we have that
			$f(ugb)=\psi_n(u)f(gb)=\psi_n(u)(b\cdot f)(g)=\psi_n(u)f(g)$. Note that the Bruhat decomposition $\P_n=\bigcup_{w\in \W}\U_nw\B_{n-1}^{\op}(k_{\F})$ implies that $f$ is uniquely determined by its values on the Weyl group $\W$. We realize $\W$ as the set of permutation matrices. For any permutation matrix $w_\sigma$, where $\sigma\in S_n$ and any matrix $A=(a_{i,j})_{1\leqslant i,j\leqslant n}$ we have that $(w_\sigma Aw_\sigma^{-1})_{i,j}=a_{\sigma(i),\sigma(j)}$. Now if $\sigma^{-1}(i)<\sigma^{-1}(i+1)$ for all $1\leqslant i\leqslant n-1$, then $\sigma=\operatorname{id}$. Hence if $w_\sigma\not=1$, there is an element $b\in \B_{n-1}^{\op}(k_{\F})$ such that $wbw^{-1}\in U_n$ and $\psi_n(wbw^{-1})\not=1$. Then 
			$f(w_\sigma)=f(w_\sigma b)=\psi_n(w_\sigma bw_\sigma^{-1})f(w_{\sigma})$ and hence 
			$f(w_\sigma)=0$ for $w_\sigma\not=1$. Hence the support of $f$ is contained in $\U_n\B_{n-1}^\op$.
			
			\textbf{ad (1):} Since we assume that there is $1\leqslant j\leqslant n-1$ such that $\alpha_{j-1}=\alpha_{j}$, we see that $\overline{g\K_n(m)g^{-1}\cap \K}$ contains the subgroup
			$$1_n+k_{\F}E_{j-1,j},$$ where $E_{j-1,j}$ is the matrix whose only nonzero entry is one at $(j-1,j)$. This implies that there exists $u\in\overline{g\K_n(m)g^{-1}\cap \K}\cap \U_n$ such that $\psi_n(u)\not=1$. We obtain 
			$$f(1)=(u\cdot f)(1)=f(u)=\psi_n(u)f(1)$$
			and hence $f(1)=0$, which implies that $f\equiv 0$ and $\Psi=0$.
			
			\textbf{ad (2):} In this case $\overline{g\K_n(m)g^{-1}\cap K}=\B_{n-1}^\op$. Now the space of $\B_{n-1}^\op$-invariant functions in $\Ind_{\U_n}^{\P_n}(\psi)$ which are supported in $\U_n\B_{n-1}^\op$ is clearly one-dimensional. Moreover, any such function gives rise to a unique morphism in 
			$$\Hom_{\B_{n-1}^\op}(1,\operatorname{ind}_{\U_n}^{\P_n}(\psi_n))$$
			which implies the result.\end{proof}
		
		By putting all the above propositions together we obtain that $\Hom_{\overline{g\K_n(m)g^{-1}\cap \K}}(1,\tau)$ can only by nonzero for $g$ in the same double coset as any of
		$$\mathcal D_n(m)\coloneqq\left\{\begin{pmatrix}
			\varpi^{\alpha_{n-1}}&&&&\\
			&\varpi^{\alpha_{n-2}}&&&\\
			&&\ddots&&\\
			&&&\varpi^{\alpha_1}&\\
			&&&&1
		\end{pmatrix}\mid\alpha_i\in\mathbb Z,0<\alpha_1<\alpha_2<\dotsc<\alpha_{n-1}<m\right\}.$$
		%\rob{However, if $m<n$ then $\mathcal D_n(m)$ is empty and if $m=n$ there is exactly one double coset for which one can show multiplicity one as in our other }

\begin{proof}[Proof of Theorem \ref{depthzerocoset}]
By the above Proposition \ref{onedim} for any element $g$ of $\mathcal D_n(m)$ the Hom-space $$\Hom_{\overline{g\K_n(m)g^{-1}\cap \K}}(1,\tau)$$
is one dimensional and we choose a nonzero $\phi_g$ in this space. By the computations in Section \ref{secmackey} any such $\phi_g$ gives rise to an element $\Phi_g$ of $\Hom_{ \R[\K_n(m)]}(1,\pi)$ and the collection $\{\Phi_{g}\mid g\in\mathcal D_n(m)\}$ is linearly independent. Hence we obtain    
 $$\dim_R(\Hom_{ \R[\K_n(m)]}(1,\pi))=|\mathcal D_n(m)|.$$ If $m<n$ then $\mathcal D_n(m)$ is empty and if~$m\geq n$ then $|\mathcal D_n(m)|= \begin{pmatrix}m-1\\n-1\end{pmatrix}$.

The explicit form of~$f_{\new}(\Sigma_n)$ follows as~$\sum_{b\in \B_{n-1}^{\op}(k_\F)}b\cdot B_{\tau,\overline{\psi}}$ is clearly~$\B_{n-1}^{\op}(k_\F)$-invariant, and it is nonzero as
\[\sum_{b\in \B_{n-1}^{\op}(k_\F)}b\cdot B_{\tau,\overline{\psi}}(1)=\sum_{b\in \B_{n-1}^{\op}(k_\F)} B_{\tau,\overline{\psi}}(b)=B_{\tau,\overline{\psi}}(1)=1\]
by property (B\ref{besselfunction1}) of Bessel functions.

It remains to show the final formula of \eqref{depthzeromain3}.  As the integral expression is clearly~$\K_n(n)$-invariant, we just need to show it is non-zero, which we do by evaluating at~$g=\Sigma_n$ and evaluating the resulting function at the identity: we consider
\[\int_{\K_n(n)} \mathcal{B}_{\widetilde{\tau}}(\Sigma_nk\Sigma_n^{-1}) dk(1)=\int_{\K_n(n)\cap \K_n^{\Sigma_n}} \mathrm{B}_{\widetilde{\tau},\overline{\psi}}(\Sigma_n k\Sigma_n^{-1}) dk.\]
As~$k\in\K_n(n)$, we can write
\[k=\begin{pmatrix}A&x \\y&z\end{pmatrix}=\begin{pmatrix}1_{n-1}&0 \\yA^{-1}& z-yA^{-1}x\end{pmatrix}\begin{pmatrix}A&x \\0&1\end{pmatrix},\]
with~$A\in\GL_{n-1}(\mathfrak{o}_\F)$, $x\in\begin{pmatrix}\mathfrak{o}_\F&\cdots &\mathfrak{o}_\F\end{pmatrix}^T$,~$y\in\begin{pmatrix}\mathfrak{p}_\F^n&\cdots &\mathfrak{p}_\F^{n}\end{pmatrix}$, and~$z\in 1+\mathfrak{p}_\F^n$.  Moreover,
\[\Sigma_n\begin{pmatrix}1_{n-1}&0 \\yA^{-1}& z-yA^{-1}x\end{pmatrix}\Sigma_n^{-1} \in 1_n+\M_n(\mathfrak{p}_\F)\subset \K_n,\]
Hence, for~$k\in\K_n(n)$ decomposed as above,
\[\Sigma_n k\Sigma_n^{-1}\in\K_n \Leftrightarrow \Sigma_n \begin{pmatrix}A&x \\0&1\end{pmatrix}\Sigma_n^{-1} \in\K_n\]
and
\[ \mathrm{B}_{\widetilde{\tau},\overline{\psi}}(\Sigma_n k\Sigma_n^{-1}) = \mathrm{B}_{\widetilde{\tau},\overline{\psi}}\left(\Sigma_n\begin{pmatrix}A&x \\0&1\end{pmatrix}\Sigma_n^{-1}\right).\]
Moreover, if~$\Sigma_n\left(\begin{smallmatrix}A&x \\0&1\end{smallmatrix}\right)\Sigma_n^{-1}\in\K_n$ then it is lower unitriangular mod~$\M_n(\mathfrak{p}_\F)$ with image inside~$\left(\begin{smallmatrix}\B_{n-1}^{\op}&0 \\0&1\end{smallmatrix}\right)$ and containing the identity.  By property (B\ref{besselfunction1}) of Bessel functions, we deduce that the integral is a non-zero constant.

\end{proof}

\section{Matrix coefficients and Whittaker functions of depth zero newform vectors}

\subsection{Matrix coefficients of depth zero newform vectors}
Suppose now~$\pi$ is a depth zero cuspidal~$\R$-representation. Choose now a Moy--Prasad--Morris--Vign\'eras model for~$\pi$ as compactly induced
\[\ind_{\Z_n\K_n}^{\G_n}(\omega_{\pi}\tau),\]
where~$\tau$ is a cuspidal representation of~$\GL_n(k_\F)$ which we identify with~$W(\tau,\overline{\psi})$ its Whittaker model, then in Theorem \ref{depthzerocoset} we showed that~$v_{\new}$ is given by the function (up to scalar)~$f_{\new}$.
  For~$\pi^{\vee}$ similarly, we take the model
\[\ind_{\Z_n\K_n}^{\G_n}(\omega_{\pi}^{-1}\tau^{\vee}),\]
where we identify~$\tau^{\vee}$ with its Whittaker model~$W(\tau^{\vee},\overline{\psi}^{-1})$ and~$f_{\new}^{\vee}$ is again given by Theorem \ref{depthzerocoset}.

Let~$dg$ be a right~$\G_n$-invariant measure on~$\Z_n\K_n\backslash \G_n$, normalized so that~$dg(\Z_n\K_n)=1$, then we can identify~$\ind_{\Z_n\K_n}^{\G_n}(\omega_{\pi}^{-1}\tau^{\vee})$ with~$(\ind_{\Z_n\K_n}^{\G_n}(\omega_{\pi}\tau))^\vee$ via the bilinear form
\begin{align*}
\ind_{\Z_n\K_n}^{\G_n}(\omega_{\pi}\tau)\times\ind_{\Z_n\K_n}^{\G_n}(\omega_{\pi}^{-1}\tau^{\vee}) &\rightarrow \R\\
(f,f')&\mapsto \int_{\Z_n\K_n\backslash \G_n}\langle f(g),f'(g)\rangle dg.
\end{align*}
We identify~$(W(\tau,\overline{\psi}))^{\vee}\simeq W(\tau^{\vee},\overline{\psi}^{-1})$ via the bilinear form
\begin{align*}
W(\tau,\overline{\psi})\times W(\tau^{\vee},\overline{\psi}^{-1})&\rightarrow \R\\
(W,W')&\mapsto \sum_{\U_n(\mathbb{F}_q)\backslash \G_n(\mathbb{F}_q)}W(k)W'(k).
\end{align*}

For~$y\in\K_n(n)$, define 
$f_{\Sigma_n y}\in\ind_{\Z_n\K_n}^{\G_n}(\omega_{\pi}\tau)$ by support~$f_{\Sigma_n y}\subseteq \Z_n\K_n\Sigma_n y$ and
$$f_{\Sigma_n y}\mid_{\Z_n\K_n\Sigma_n y}=f_{\new}\mid_{\Z_n\K_n\Sigma_n y}$$
and analogously
$f_{\Sigma_n y}^\vee\in\ind_{\Z_n\K_n}^{\G_n}(\omega_{\pi}^{-1}\tau^{\vee})$ by support~$f_{\Sigma_n y}^{\vee}\subseteq \Z_n\K_n\Sigma_n y$ and
\[f_{\Sigma_n y}^{\vee}\mid_{\Z_n\K_n\Sigma_n y}=f_{\new}^{\vee}\mid_{\Z_n\K_n\Sigma_n y}.\]
Let $y_1,\dotsc,y_r$ be elements of $\K_n(n)$ such that $\Sigma_n y_1,\dotsc,\Sigma_n y_r$ are representatives for $\K_n\Z_n\backslash \K_n\Z_n\Sigma_n \K_n(n).$ Then
\[f_\new=\sum_{i=1}^rf_{\Sigma_n y_i},\text{ and }f_\new^\vee=\sum_{i=1}^rf_{\Sigma_n y_i}^\vee.\]
Note that by right~$\K_n(n)$-invariance of~$f_{\new}^{\vee}$, we have~$f_{\Sigma_n }(x)=f_{\Sigma_n y}(xy)$, for~$y\in\K_n(n)$.  

\begin{remark}
A priori we have so far obtained two expressions for the newform in $\ind_{\Z_n\K_n}^{\G_n}(\widetilde{\tau})$. Firstly, $f_\new=\sum_{i=1}^rf_{\Sigma_n y_i}$ and furthermore
$$f_{\new}(g)=\int_{\K_n(n)} \mathcal{B}_{\widetilde{\tau}}(gk\Sigma_n^{-1}) dk.$$
We will show that one can obtain the first expression by simply evaluating the above integral.

 As above let $\Sigma_n y_1,\dotsc,\Sigma_n y_r\in \K_n(n)$ be a set of coset representatives for $\K_n\Z_n\backslash \K_n\Z_n\Sigma_n \K_n(n)$, which implies that $y_1,\dotsc,y_r\in \K_n(n)$ is a set of coset representatives for $\K_n(n)\cap \K_n^{\Sigma_n}\backslash \K_n(n)$. In particular we obtain that 
$$\int_{\K_n(n)} \mathcal{B}_{\widetilde{\tau}}(gk\Sigma_n^{-1}) dk=\sum_{i=1}^r\int_{\K_n(n)\cap \K_n^{\Sigma_n}} \mathcal{B}_{\widetilde{\tau}}(gy_i^{-1}k\Sigma_n^{-1}) dk=\sum_{i=1}^r\int_{\K_n(n)^{\Sigma_n^{-1}}\cap \K_n} \mathcal{B}_{\widetilde{\tau}}(g(\Sigma_n y_i)^{-1}k) dk$$
and we claim that $\int_{\K_n(n)^{\Sigma_n^{-1}}\cap \K_n} \mathcal{B}_{\widetilde{\tau}}(g(\Sigma_n y_i)^{-1}k) dk=f_{\Sigma_n y_i}(g)$ for all $g\in\G_n$. Since $\mathrm{supp}(\mathcal{B}_{\widetilde{\tau}})\subseteq\Z_n\K_n$ we obtain that if $$\int_{\K_n(n)^{\Sigma_n^{-1}}\cap \K_n} \mathcal{B}_{\widetilde{\tau}}(g(\Sigma_n y_i)^{-1}k) dk\not=0$$ then $g\in \K_n\Z_n\Sigma_n y_i$. Suppose that $g\in\K_n\Z_n\Sigma_n y_i$ and write $g=zk'\Sigma_n y_i$ where $k'\in\K_n$ and $z\in\Z_n$. We obtain
$$\int_{\K_n(n)^{\Sigma_n^{-1}}\cap \K_n} \mathcal{B}_{\widetilde{\tau}}(g(\Sigma_n y_i)^{-1}k) dk=\omega(z)\int_{\K_n(n)^{\Sigma_n^{-1}}\cap \K_n} \mathcal{B}_{\widetilde{\tau}}(k'k)dk. $$
If $x\in \K_n(n)^{\Sigma_n^{-1}}\cap \K_n^1$ we have that $\mathcal{B}_{\widetilde{\tau}}(k'kx)=\mathcal{B}_{\widetilde{\tau}}(k'k)$ for all $k\in\K_n(n)^{\Sigma_n^{-1}}\cap \K_n$ and it is straightforward to see that $\K_n(n)^{\Sigma_n^{-1}}\cap \K_n/\K_n(n)^{\Sigma_n^{-1}}\cap \K_n^1\cong \B_{n-1}^{\op}(k_\F)$. Hence we obtain that (we assume the measure of $\K_n(n)^{\Sigma_n^{-1}}\cap \K_n^1$ to be $1$)
$$\omega(z)\int_{\K_n(n)^{\Sigma_n^{-1}}\cap \K_n} \mathcal{B}_{\widetilde{\tau}}(k'k)dk=\omega(z)\sum_{x\in\B_{n-1}^{\op}(k_\F)}\overline{k'}x\cdot B_{\tau,\overline{\psi}},$$
which agrees with $f_{\Sigma_n y_i}(g)$.
\end{remark}
The matrix coefficient~$c_{\mathcal{B}_{\widetilde{\tau}},\mathcal{B}_{\widetilde{\tau}^{\vee}}}$, has a simple description: $\mathrm{supp}(c_{\mathcal{B}_{\widetilde{\tau}},\mathcal{B}_{\widetilde{\tau}^{\vee}}})\subseteq \Z_n\K_n$ and, for~$k\in\K_n$,~$z\in\Z_n$, expanding the definitions we have
\[c_{\mathcal{B}_{\widetilde{\tau}},\mathcal{B}_{\widetilde{\tau}^{\vee}}}(zk)=\omega_{\pi}(z)B_{\tau,\psi}(k).\]
This is a special case of the Bessel functions of \cite[Proposition 5.7]{PaskSte}.

We now give a formula for the canonical bi-$\K_n(n)$-invariant matrix coefficient:
\begin{proposition}\label{matrixcoeffsdepthzero}
Let $y_1,\dotsc,y_r$ be a set of coset representatives of $\K_n\Z_n\backslash \K_n\Z_n\Sigma_n \K_n(n)$.  The matrix coefficient~$c_{f_{\new},f_{\new}^{\vee}}$ is non-zero and satisfies:
\begin{enumerate}
\item (Bi-$\K_n(c(\pi))$-invariance),~i.e.,~$c_{f_{\new},f_{\new}^{\vee}}\in(\ind_{\Z_n}^{\G_n}(\omega_{\pi}))^{\K_n(c(\pi))\times\K_n(c(\pi))}$;

\item  $\mathop{supp}(c_{f_{\new},f_{\new}^{\vee}})\subseteq \Z_n\K_n(n)(\K_n^{\Sigma_n})\K_n(n)$ and for~$g\in \K_n^{\Sigma_n}$ let \[I_{g}=\{(i,j)\in\mathbb Z^2\mid 1\leqslant i,j\leqslant r\text{ and }g\in y_j^{-1}\Z_n(\K_n^{\Sigma_n})y_i\},\]
	 then we have
	\[c_{f_\new,f_{new}^\vee}(g)=\sum_{(i,j)\in I_g}\frac{|\G_n(\mathbb{F}_q)|}{|\U_n(\mathbb{F}_q)|\dim(\tau)}\sum_{b,b'\in \B_{n-1}^{\op}(\mathbb{F}_q)}B_{\tau,\overline{\psi}}(b\Sigma_n y_jgy_i^{-1}\Sigma_n^{-1}b').\]
	\item (expression in terms of~$c_{\mathcal{B}_{\widetilde{\tau}},\mathcal{B}_{\widetilde{\tau}^{\vee}}}$) for~$g\in\G_n$,
	\[c_{f_{\new},f_{\new}^{\vee}}(g)=\int_{\K_n(n)}\int_{\K_n(n)}c_{\mathcal{B}_{\widetilde{\tau}},\mathcal{B}_{\widetilde{\tau}^{\vee}}}(k'\Sigma_n gk\Sigma_{n}^{-1})dk'dk.\]
	\end{enumerate}
\end{proposition}

\begin{proof}
The final part follows directly from Theorem \ref{depthzerocoset} \eqref{depthzeromain3}.  Let $x,y\in \K_n(n)$. By our identification,
\[c_{f_{\Sigma_n x},f_{\Sigma_n y}^{\vee}}(j)=\int_{\Z_n\K_n\backslash \G_n}\langle f_{\Sigma_n x}(gj),f_{\Sigma_n y}^\vee(g)\rangle dg.\]
However, $f_{\Sigma_n y}^\vee$ is supported in $\K_n\Z_n\Sigma_n y$ and note that for all $g\in \K_n\Z_n\Sigma_n y$ we have
\[\langle f_{\Sigma_n x}(gj),f_{\Sigma_n y}^\vee(g)\rangle=\langle f_{\Sigma_n x}(\Sigma_n yj),f_{\Sigma_n y}^\vee(\Sigma_n y)\rangle.\]
Hence 
\[c_{f_{\Sigma_n x},f_{\Sigma_n y}^\vee}(j)=\mu_{\K_n\Z_n\backslash \G_n}(\K_n\Z_n\Sigma_n y)\langle f_{\Sigma_n x}(\Sigma_n yj),f_{\Sigma_n  y}^\vee(\Sigma_n y)\rangle.\]
Now for $f_{\Sigma_n x}(\Sigma_n yj)$ to be nonzero we need to have that $j\in y^{-1}\Z_n(\K_n^{\Sigma_n})x$ and hence 
\begin{equation}\label{eq:suppcoeff}
	\operatorname{supp}(c_{f_{\Sigma_n x},f_{\Sigma_n y}^\vee})\subseteq y^{-1}Z(\K_n^{\Sigma_n})x.
\end{equation} Let $j\in y^{-1}\K_n^{\Sigma_n} x$ and write $j'=\Sigma_n yjx^{-1}\Sigma_n^{-1}$.
By the second identification we have
\begin{align*}
	\langle f_{\Sigma_n x}(\Sigma_n yj),f_{\Sigma_n y}^\vee(\Sigma_n y)\rangle&=\sum_{k\in \U_n(\mathbb{F}_q)\backslash \G_n(\mathbb{F}_q)} \sum_{b,b'\in\B_{n-1}^{\op}(\mathbb{F}_q)}B_{\tau,\overline{\psi}}(kj'b) B_{\tau^{\vee},\overline{\psi}^{-1}}(kb')\\
	&=|\U_n(\mathbb{F}_q)|^{-1} \sum_{b,b'\in\B_{n-1}^{\op}(\mathbb{F}_q)} \left(\sum_{k\in \G_n(\mathbb{F}_q)}B_{\tau,\overline{\psi}}(kb'j'b) B_{\tau^{\vee},\overline{\psi}^{-1}}(k)\right).
\end{align*}
By the formula for the Bessel function: $B_{\tau,\psi}(a)=|\U_n(\mathbb{F}_q)|^{-1}\sum_{u\in\U_n(\mathbb{F}_q)}\psi^{-1}(u)\chi_{\tau}(au)$ of (B\ref{besselfunction2}) and orthogonality of characters we see that
\begin{align*}
			\langle f_{\Sigma_n x}(\Sigma_n yj),f_{\Sigma_n y}^\vee(\Sigma_n y)\rangle&=|\U_n(\mathbb{F}_q)|^{-3} \sum_{b,b'\in\B_{n-1}^{\op}(\mathbb{F}_q)} \left(\sum_{k\in G(\mathbb{F}_q)}\sum_{u,u'\in \U_n(\mathbb{F}_q)}\psi^{-1}(u)\chi_{\tau}(kb'j'bu) \psi(u')\chi_{\tau^{\vee}}(ku')\right)\\
	&=|\U_n(\mathbb{F}_q)|^{-3} \sum_{b,b'\in\B_{n-1}^{\op}(\mathbb{F}_q)} \sum_{u,u'\in \U_n(\mathbb{F}_q)}\psi(u^{-1}u')\left(\sum_{k\in\G_n(\mathbb{F}_q)}\chi_{\tau}(ku'^{-1}b'j'bu)\chi_{\tau^{\vee}}(k)\right)\\
	&=\frac{|\G_n(\mathbb{F}_q)|}{|\U_n(\mathbb{F}_q)|^{3}\dim(\tau)} \sum_{b,b'\in\B_{n-1}^{\op}(\mathbb{F}_q)} \sum_{u,u'\in \U_n(\mathbb{F}_q)}\psi(u^{-1}u')\chi_{\tau}(u'^{-1}b'j'bu).
\end{align*}
Hence we can compute
\begin{align*}
	\langle f_{\Sigma_n x}(\Sigma_n y j),f_{\Sigma_n y}^\vee(\Sigma_n y)\rangle&=\frac{|\G_n(\mathbb{F}_q)|}{|\U_n(\mathbb{F}_q)|^{3}\dim(\tau)}\sum_{b,b'\in\B_{n-1}^{\op}(\mathbb{F}_q)} \sum_{u,u'\in \U_n(\mathbb{F}_q)}\psi(u'^{-1}u^{-1}u')\chi_{\tau}(u'^{-1}b'j'buu')\\
	&=\frac{|\G_n(\mathbb{F}_q)|}{|\U_n(\mathbb{F}_q)|^{2}\dim(\tau)}\sum_{b,b'\in\B_{n-1}^{\op}(\mathbb{F}_q)} \sum_{u\in \U_n(\mathbb{F}_q)}\psi(u^{-1})\chi_{\tau}(b'j'bu) \\
	&=\frac{|\G_n(\mathbb{F}_q)|}{|\U_n(\mathbb{F}_q)|\dim(\tau)}\sum_{b,b'\in\B_{n-1}^{\op}(\mathbb{F}_q)}B_{\tau,\psi}(b'j'b). \\
	&=\frac{|\G_n(\mathbb{F}_q)|}{|\U_n(\mathbb{F}_q)|\dim(\tau)}\sum_{b,b'\in\B_{n-1}^{\op}(\mathbb{F}_q)}B_{\tau,\psi}(b'\Sigma_n yjx^{-1}\Sigma_n^{-1}b).
\end{align*}
Let $y_1,\dotsc,y_r\in \K_n(n)$ be a set of coset representatives for $\K_n\Z_n\backslash \K_n\Z_n\Sigma_n \K_n(n)$ and choose $y_1$ to be the identity. As 
$f_\new=\sum_{i=1}^rf_{\Sigma_n y_i}$, we see that
\[c_{f_\new,f_{\Sigma_n}^\vee}(x)=\sum_{i=1}^rc_{f_{\Sigma_n y_i},f_{\Sigma_n}^\vee}(x)\]
and hence $\operatorname{supp}(c_{f_\new,f_{\Sigma_n}^\vee})\subseteq \Z_n(\K_n^{\Sigma_n})\K_n(n)$.
For $g\in \K_n^{\Sigma_n}$ note that $c_{f_{\Sigma_n y_i},f_{\Sigma_n}^\vee}(g)$ is nonzero if and only if $y_i$ is the identity, i.e.\ $i=1$. We obtain that
\[c_{f_{\new},f_{\Sigma_n}^{\vee}}(g)=\frac{|\G_n(\mathbb{F}_q)|}{|\U_n(\mathbb{F}_q)|\dim(\tau)}\sum_{b,b'\in \B_{n-1}^{\op}(\mathbb{F}_q)}B_{\tau,\overline{\psi}}(b\Sigma_n g\Sigma_n^{-1}b').\]
Moreover, we have that
\[c_{f_\new,f_{\new}^\vee}(g)=\sum_{i,j=1}^rc_{f_{\Sigma_n y_i},f_{\Sigma_n {y_j}}^\vee}(g)\]
and hence $\operatorname{supp}(c_{f_\new,f_\new^\vee})\subseteq \Z_n\K_n(n)(\K_n^{\Sigma_n})\K_n(n)$. For $g\in \K_n^{\Sigma_n}$ let $$I_{g}=\{(i,j)\in\mathbb Z^2\mid 1\leq i,j\leq r\text{ and }g\in y_j^{-1}\Z_n(\K_n^{\Sigma_n})y_i\}$$ 
and then 
$$c_{f_\new,f_{\new}^\vee}(g)=\sum_{(i,j)\in I_g}\frac{|\G_n(\mathbb{F}_q)|}{|\U_n(\mathbb{F}_q)|\dim(\tau)}\sum_{b,b'\in \B_{n-1}^{\op}(\mathbb{F}_q)}B_{\tau,\overline{\psi}}(b\Sigma_n y_jgy_i^{-1}\Sigma_n^{-1}b').$$

\end{proof}

While the first formula for the canonical bi-$\K_n(n)$-invariant matrix coefficient is only explicit up to the sets~$I_g$, our proof allows us to give an explicit formula for the~$\K_n(n)$-invariant matrix coefficient~$c_{f_{\new},f_{\Sigma_n}^{\vee}}$:
\begin{cor}
The matrix coefficient~$c_{f_{\new},f_{\Sigma_n}^{\vee}}$ satisfies:
\begin{enumerate}
\item (Right~$\K_n(c(\pi))$-invariance),~i.e.,~$c_{f_{\new},f_{\Sigma_n}^{\vee}}\in(\ind_{\Z_n}^{\G_n}(\omega_{\pi}))^{\K_n(c(\pi))}$;
\item $\mathop{supp}(c_{f_{\new},f_{\Sigma_n}^{\vee}})\subseteq \Z_n\K_n^{\Sigma_n}\K_n(n)$ and for~$g\in\K_n^{\Sigma_n}$. 
\[c_{f_{\new},f_{\Sigma_n}^{\vee}}(j)=\frac{|\G_n(\mathbb{F}_q)|}{|\U_n(\mathbb{F}_q)|\dim(\tau)}\sum_{b,b'\in \B_{n-1}^{\op}(\mathbb{F}_q)}B_{\tau,\overline{\psi}}(b\Sigma_n g\Sigma_n^{-1}b').\]
In particular,~$c_{f_{\new},f_{\Sigma_n}^{\vee}}(1)=\frac{|\G_n(\mathbb{F}_q)|}{|\U_n(\mathbb{F}_q)|\dim(\tau)}$.
\end{enumerate}
%\begin{enumerate}

\end{cor}

\subsection{Whittaker newforms of depth zero cuspidal representations}
Given~$\psi\colon \F\to\mathbb \R^\times$ a non-trivial character, we define a non-degenerate character $\psi:\U_n\rightarrow\R^\times$ in the usual way:
\[ \psi(u)=\psi(\sum_{i=1}^{n-1} u_{i,i+1}).\]   
Using finite Bessel functions, Gelfand constructs an explicit Whittaker function for a depth zero cuspidal representation with small support:

\begin{proposition}[Gelfand]
Let~$\pi$ be a depth zero cuspidal~$\R$-representation of~$\G_n$ containing cuspidal type~$(\K_n,\tau)$, and suppose~$\psi:\F\rightarrow \R^\times$ has conductor~$\mathfrak{p}_\F$.  Then there is a unique Whittaker function~$\W_{\pi,\mathrm{Gel},\psi}\in\W(\pi,\psi)$ defined by:
\begin{align*}
\mathrm{supp}(\W_{\pi,\mathrm{Gel},\psi})&\subseteq \Z_n\U_n\K_n\\
\W_{\pi,\mathrm{Gel},\psi}(zuk)&=\omega_{\pi}(z)\psi(u) B_{\tau,\overline{\psi}}(\overline{k}).
\end{align*}
\end{proposition}

Suppose that $\psi'$ has conductor $\mathfrak p_F$.  Note that the associated non-degenerate character of~$\U_n$ defined by $\psi'^{\Sigma_n^{-1}}$ is given by a character~$\psi$ of conductor $\mathfrak o_F$.  Hence for $f\in\W(\pi,\psi')$ the map $\tilde{f}(g)\coloneqq f(\Sigma_ng)$ is an element of $\Ind_{\U_n}^{\G_n}(\psi)$. Since $f\mapsto\tilde{f}$ is $\G_n$-equivariant we see that $\tilde{f}\in\W(\pi,\psi)$. Hence $g\mapsto\W_{\pi,\mathrm{Gel},\psi'}(\Sigma_ng)$ is an element of $W(\pi,\psi)$ and also for any~$k'\in \K_n(n)$ the function $g\mapsto(k'\Sigma_n^{-1})\cdot\W_{\pi,\mathrm{Gel},\psi'}(\Sigma_ng)$. This shows that the function 
\[g\mapsto\int_{\K_n(n)}  \W_{\pi,\mathrm{Gel},\psi^{\Sigma_n^{-1}}}(\Sigma_n g k\Sigma_n^{-1})  dk\]
is also an element of $\W(\pi,\psi)$.

In the next proposition, we give a strong bound on the support of the Whittaker newform of a depth zero cuspidal~$\R$-representation, and two integral expressions for it -- one given by a Jacquet integral of the newform matrix coefficient, and the second by integrating the translated-conjugate average of the Gelfand's Whittaker function over~$\K_n(n)$ defined above:
\begin{proposition}\label{depthzerowhittakernewform}
Let~$\pi$ be a depth zero cuspidal~$\R$-representation of~$\G_n$, and suppose~$\psi:\F\rightarrow \mathbb{C}^\times$ has conductor~$\mathfrak{o}_\F$.  The Whittaker newform~$\W_{\pi,\new,\psi}\in\W(\pi,\psi)$ (normalized at the identity) satisfies the following support condition
\[\mathrm{supp}(\W_{\pi,\new,\psi})\subseteq \Z_n\U_n\K_n^{\Sigma_n}\K_n(n),\]
and for~$g\in\G_n$ is given by
\[\W_{\pi,\new,\psi}(g)=\int_{\U_n}\psi^{-1}(u)c_{f_\new,f_\new^\vee}(ug)du,\]
where~$du$ is the~$\R$-Haar measure on~$\U_n$ normalized by~$du(\U\cap (\K)^{\Sigma})=1$.  Moreover, for~$g\in\G_n$, 
\[ \W_{\pi,\new,\psi}(g)=\int_{\K_n(n)}  \W_{\pi,\mathrm{Gel},\psi^{\Sigma_n^{-1}}}(\Sigma_n g k\Sigma_n^{-1})  dk\]
for an appropriately normalized~$\R$-Haar measure~$dk$ on~$\K_n(n)$.  
\end{proposition}
\begin{proof}
The support condition will follow once we have established either integral expression -- directly from the support of our matrix coefficient or of Gelfand's Whittaker function. For the first expression note that the map
\begin{align*}
	\Xi\colon\pi&\to\Ind_{\U_n}^{\G_n}(\psi_n)\\
	v&\mapsto\left(g\mapsto \int_{\U_n}\psi^{-1}(u)c_{v,f_\new^\vee}(ug)du\right),
	\end{align*}
is clearly an intertwining map. Since $\pi$ is irreducible and $\Xi(f_\new)$ is $K(n)$-invariant the formula follows if we can show that $\Xi$ is nonzero. In particular it is enough to show that
$$\int_{\U_n}\psi^{-1}(u)c_{f_\Sigma,f_\new^\vee}(u)du\not=0.$$ 
We write~$c$ for the matrix coefficient~$c_{f_\Sigma,f_\new^\vee}$. By Equation (\ref{eq:suppcoeff}) we see that $\operatorname{supp}(c)\subseteq \K_n(n)\Z_n(\K_n^{\Sigma_n}) $ and hence the integral is actually over $\U_n\cap \K_n(n)\Z_n(\K_n^{\Sigma_n})$. We have that 
\[\U_n\cap \K_n(n)\Z_n(\K_n^{\Sigma_n})=\U_n\cap \K_n^{\Sigma_n}\]
which in particular implies that
$$\int_{\U_n}\psi^{-1}(u)c(u)du=\int_{\U_n}\psi^{-1}(u)c_{f_\Sigma,f_\Sigma^\vee}(u)du.$$
Note that by our choices $\psi^{-1}(u'u)c(u'u)=\psi^{-1}(u)c(u)$ for all $u\in \U\cap \K_n^{\Sigma_n},u'\in \U\cap (\K_n^1)^{\Sigma}$. Hence we need to compute
\[\sum_{u\in \U\cap (\K_n^1)^{\Sigma}\backslash \U\cap \K_n^{\Sigma_n}}\psi^{-1}(u)c(u).\]
However, note that $\U\cap (\K_n^1)^{\Sigma}\backslash \U\cap \K_n^{\Sigma_n}\cong \U_n(k_\F)$, and we find
\[\sum_{u\in \U\cap (\K_n^1)^{\Sigma}\backslash \U\cap \K_n^{\Sigma_n}}\psi^{-1}(u)c(u)=\sum_{u\in \U_n(k_\F)}\overline{\psi}^{-1}(u)\sum_{b,b'\in \B_{n-1}^{\op}(k_\F)}B_{\tau,\overline{\psi}}(bub').\]
To compute this, we will use the following lemma:
\begin{lemma}
	Let $\alpha\colon\GL_r(k_\F)\to\mathbb C$ be a function such that $\alpha(ug)=\overline{\psi}(u)\alpha(g)$ for all $u\in \U_r(k_\F),g\in\GL_r(k_\F)$. Then for all $g\in\GL_r(k_\F)$ we have that 
	$$\frac{1}{|\U_r(k_\F)|}\sum_{u\in \U_r(k_\F),b\in \B_{r-1}^{\op}(k_\F)}\overline{\psi}^{-1}(u)\alpha(bug)=\alpha(g).$$
\end{lemma}
\begin{proof} 
	We proceed via induction on $r$. For $x\in k_\F^{r-1}$ let $n(x)$ be the matrix in $\U_r(k_\F)$ defined by	
	\[n(x)=\begin{pmatrix}
		1_{r-1}&x\\0&1
	\end{pmatrix}.\]	
	Let $\N_r(k_\F)$ be the abelian subgroup of $\U_r(k_\F)$ consisting of matrices of the form $n(x)$ where $x\in k_\F^{r-1}$. Note that $\U_r(k_\F)=\U_{r-1}(k_\F)\ltimes \N_r(k_\F)$, and hence we have
	\[\sum_{u\in \U_r(k_\F),b\in \B_{r-1}^{\op}(k_\F)}\alpha(bug)=\sum_{u\in \U_{r-1}(k_\F)}\overline{\psi}^{-1}(u)\sum_{b\in \B_{r-1}^{\op}(k_\F),x\in \N_r(k_\F)}\overline{\psi}^{-1}(x)\alpha(bxug).\]
	Now $bxb^{-1}\in \N_{r}(k_\F)$ and hence  
	\[\sum_{x\in \N_r(k_\F)}\overline{\psi}^{-1}(x)\alpha(bxug)=\sum_{x\in \N_r(k_\F)}\overline{\psi}(x^{-1}bxb^{-1})\alpha(bug).\]
	Let $x=n(x')$ for some $x'\in k_\F^{r-1}$. Then $x^{-1}bxb^{-1}=n((b-I_{r-1})x')$.  Now if the bottom row of $b$ is not equal to $(0,\dotsc,0,1)$ then~$\sum_{x\in \N_r(k_\F)}\overline{\psi}(x^{-1}bxb^{-1})=0$.  If $b$ has bottom row equal to $(0,\dotsc,0,1)$, i.e.\ $b\in \B_{r-2}^\op(k_\F)$, then $\overline{\psi}(x^{-1}bxb^{-1})=1$. Hence
	\[\sum_{b\in \B_{r-1}^{\op}(k_\F),x\in \N_r(k_\F)}\overline{\psi}^{-1}(x)\alpha(bxug)=|\N_r(k_\F)|\sum_{b\in \B_{r-2}^{\op}(k_\F)}\alpha(bug).\]
	Since $|\U_r(k_\F)|=|\U_{r-1}(k_\F)|\cdot|\N_r(k_\F)|$, 
	this implies that 
	$$\frac{1}{|U_r(k_\F)|}\sum_{u\in \U_r(k_\F), b\in \B_{r-1}^{\op}(k_\F)}\overline{\psi}^{-1}(u)\alpha(bug)=\frac{1}{|\U_{r-1}(k_\F)|}\sum_{u\in \U_{r-1}(k_\F), b\in \B_{r-2}^{\op}(k_\F)}\overline{\psi}^{-1}(u)\alpha(bug)$$
	and we obtain the result via induction.
\end{proof}
Hence, we see that
\[\sum_{u\in \U(k_\F)}\overline{\psi}^{-1}(u)\sum_{b,b'\in \B_{n-1}^{\op}(k_\F)}B_{\tau,\overline{\psi}}(bub')=|\U_{n}(k_\F)|\sum_{b'\in \B_{n-1}^{\op}(k_\F)}B_{\tau,\overline{\psi}}(b')=|\U_n(k_\F)|,\]
which establishes our first integral expression. Our second integral expression follows from Theorem \ref{depthzerocoset} \eqref{depthzeromain3}.
\end{proof}

\section{Newform vectors in minimax cuspidal representations}
\subsection{Parahoric subgroups and their filtrations}
Let~$\V$ be an~$\F$-vector space of dimension~$n$,~$\G=\GL_{\F}(\V)$, and~$\A=\End_{\F}(\V)$.  Fix a character~$\psi:\F\rightarrow \R^\times$ of conductor~$\mathfrak{p}_\F$ (i.e., trivial on~$\mathfrak{p}_{\F}$, but not on~$\mathfrak{o}_{\F}$).   

An~$\mathfrak{o}_\F$-lattice in~$\V$ is a compact open~$\mathfrak{o}_\F$-submodule of~$\V$.  Given any~$\mathfrak{o}_\F$-lattice~$L$ there exists a basis~$\{e_1,\ldots,e_n\}$ of~$\V$ such that~$L=\mathfrak{o}_\F e_1 \oplus \cdots \oplus \mathfrak{o}_\F e_n$.   Let~$\mathrm{Latt}(\V)$ denote the set of~$\mathfrak{o}_\F$-lattices in~$\V$.  A function~$\Lambda:\mathbb{Z}\rightarrow \mathrm{Latt}(\V)$ is called a \emph{lattice chain} if
\begin{enumerate}
\item it is \emph{strictly decreasing}: $\Lambda(i)\subset \Lambda(i-1)$ for all~$i\in\mathbb{Z}$;
\item \emph{periodic}: there exists~$e(\Lambda)\in\mathbb{N}$ such that~$\Lambda(i+e(\Lambda))=\mathfrak{p}_\F\Lambda(i)$ for all $i\in\mathbb{Z}$.
\end{enumerate}
 Given a lattice chain in~$\V$, there exists a basis~$\{v_1,\ldots v_n\}$ of~$\V$ such that for~$0\leqslant i\leqslant e(\Lambda)-1$ we have
\[\Lambda(i)=\mathfrak{o}_\F v_1 \oplus \cdots \oplus \mathfrak{o}_\F v_{r_i}\oplus \mathfrak{p}_{\F} v_{r_i+1}\oplus \cdots \oplus \mathfrak{p}_\F v_n;\]
we call such a basis \emph{standard}.  We say that a lattice chain~$\Lambda$ is principal if~$\dim_{k_\F}(\Lambda(i)/\Lambda(i+1))$ is independent of~$i$.  

Given a lattice chain~$\Lambda$, the submodule of~$\A$ defined by
\[\mathfrak{A}_0(\Lambda)=\bigcap_{i=0}^{e(\Lambda)-1}\End_{\mathfrak{o}_\F}(\Lambda(i)),\]
is a hereditary~$\mathfrak{o}_\F$-order in~$\A$, and all hereditary~$\mathfrak{o}_\F$-orders in~$\A$ arise in this way.  With respect to a standard basis for~$\Lambda$,~$\mathfrak{A}_0(\Lambda)$ is contained in~$\mathrm{M}_n(\mathfrak{o}_\F)$ and block-upper-triangular modulo~$\mathfrak{p}_\F$.  If~$\Lambda$ is principal then the blocks are of the same size equal to~$n/e(\Lambda)$.

The hereditary order~$\mathfrak{A}_0(\Lambda)$ has a decreasing filtration
\[\mathfrak{A}_r(\Lambda)=\{x\in\A: x\Lambda(i)\subseteq \Lambda(i+k)\text{ for all }i\in\mathbb{Z}\},\]
by~$\mathfrak{o}_\F$-submodules, and~$\mathfrak{P}=\mathfrak{P}(\Lambda)=\mathfrak{A}_1(\Lambda)$ is the Jacobson radical of~$\mathfrak{A}_0(\Lambda)$.  We set
\[\U^0(\Lambda)=\mathfrak{A}_0(\Lambda)^\times,\quad \U^r(\Lambda)=1+\mathfrak{A}_r(\Lambda),\]
then~$\U^0(\Lambda)$ is a \emph{parahoric subgroup} of~$\G$, with associated filtration~$\U^r(\Lambda)$.  The commutator subgroup~$[\U^r(\Lambda),\U^s(\Lambda)]$ is contained in $\U^{r+s}(\Lambda)$ and we have the following crucial isomorphism:
\begin{lemma}
For integers~$0\leqslant r\leqslant s\leqslant 2r+1$, we have an isomorphism
\begin{align*}
\mathfrak{P}^{-s}/\mathfrak{P}^{-r}&\rightarrow \Hom(\U^{r+1}(\Lambda)/\U^{s+1}(\Lambda),\R^\times)\\
b+\mathfrak{P}^{-r}&\mapsto [\psi_{b}:x\mapsto \psi(\mathrm{Tr}_{\A/\F}(b (x-1)))].
\end{align*}
\end{lemma}
\subsection{Strata in cuspidal representations}\label{ss:strata}

A \emph{stratum} in~$\A$ is a $4$-tuple~$[\Lambda,r,s,\beta]$ where
\begin{enumerate}
\item $\Lambda$ is an~$\mathfrak{o}_\F$-lattice chain;% hereditary~$\mathfrak{o}_{\F}$-order in~$\A$, whose Jacobson radical we denote by~$\mathfrak{P}$;
\item $s<r$ are integers;
\item $\beta\in\mathfrak{P}^{-r}(\Lambda)$.
\end{enumerate}
A \emph{stratum}~$[\Lambda,m,m-1,\beta]$ in~$\A$ is called \emph{fundamental} if~$\beta+\mathfrak{P}^{-m+1}(\Lambda)$ contains no nilpotent elements.

Let~$\pi$ be an irreducible representation of~$\G$ which is not of depth zero.  Then there exists a fundamental stratum~$[\Lambda,m,m-1,\beta]$ with~$r\geqslant 1$ such that
\begin{equation}\tag{\dag}
\label{contain}
\Hom_{\U^m(\Lambda)}(\psi_\beta,\pi)\neq 0.\end{equation}
Following Bushnell and Kutzko, we say that~$\pi$ \emph{contains} the stratum~$[\Lambda,m,m-1,\beta]$.  The rational number~$m/e(\Lambda)$ is an invariant of~$\pi$ (i.e., does not depend on the choice of fundamental stratum), which we call the \emph{depth} of~$\pi$, and we write~$d(\pi)=m/e(\Lambda)$.

For a cuspidal representation of~$\G$ (of positive depth), Bushnell and Kutzko show that we can choose a particularly nice fundamental stratum -- where, for example, amongst its additional properties~$\beta$ generates a field in~$\A$ -- satisfying \eqref{contain}.  

A \emph{minimal} stratum in~$\A$ is a fundamental stratum~$[\Lambda,m,m-1,\beta]$ satisfying:
\begin{enumerate}
\item the~$\F$-algebra~$\E=\F[\beta]$ is a field in~$\A$; 
\item $\Lambda$ is an~$\mathfrak{o}_\E$-lattice chain;
\item the element~$\beta$ is \emph{minimal} over~$\F$, meaning:
\begin{enumerate}
\item $\mathrm{gcd}(m,e(\E/\F))=1$;
\item $\varpi^{m}\beta^{e(\E/\F)}+\mathfrak{p}_\E$ generates the extension of residue fields~$k_\E/k_\F$.
\end{enumerate}
\end{enumerate}
Note that,~$e(\Lambda)=e(\E/\F)$.

A minimax stratum~$[\Lambda,m,m-1,\beta]$ in~$\A$ is a minimal stratum which is also \emph{maximal} in that~$\deg(\E/\F)=\deg(\F[\beta]/\F)=n$.  In this case, up to translation, there is a unique choice for~$\Lambda$ given by~$\{\mathfrak{p}_{\E}^i:i\in\mathbb{Z}\}$.   

\begin{definition}
A (positive depth) cuspidal~$\R$-representation of~$\G$ is called \emph{minimax} if it \emph{contains} a minimax stratum~$[\Lambda,m,m-1,\beta]$ such that~$\Hom_{\U^{\lfloor m/2\rfloor +1}(\Lambda)}(\psi_{\beta},\pi)\neq 0$.  (In other words, it also contains the simple stratum~$[\Lambda,m,0,\beta]$ following Bushnell and Kutzko's terminology.)
\end{definition}
\begin{remark}
Every minimax cuspidal~$\R$-representation is supercuspidal -- cf., \cite[III 5.14]{Vig96}.
\end{remark}

\subsection{Minimax characters}
Let~$[\Lambda,m,0,\beta]$ be a simple stratum in~$\A$ such that~$[\Lambda,m,m-1,\beta]$ is a minimax stratum in~$\A$. We identify~$\V$ with~$\E=\F[\beta]$ as an~$\F$-vector space via the choice of (ordered)~$\F$-basis
\[\mathcal{B}=\{1,\beta,\beta^2,\ldots,\beta^{n-1}\}.\]
We write~$e=e(\E/\F)$.  With respect to the basis~$\mathcal{B}$, which we call the companion basis to~$[\Lambda,m,0,\beta]$,~$\beta$ is in companion matrix form:
\[\beta=\begin{pmatrix}  &&&-a_0\\
1&&&-a_1\\
&\ddots&&\vdots\\
&&1&-a_{n-1}\end{pmatrix}\]
where~$f_{\beta}(X)=X^{n}+a_{n-1}X +\cdots +a_0$ is the minimal polynomial of~$\beta$.  We have
\[\nu_{\F}(a_0)=-mn/e,\quad \nu_\F(a_i)\geqslant -m(n-i)/e.\]

This basis has the nice property that the restriction of~$\psi_{\beta}$ to upper triangular unipotent matrices~$\U_n$ (intersected with~$\U^{\lfloor m/2\rfloor +1}(\Lambda)$), agrees with the standard non-degenerate character~${\Psi}(x)=\psi(x_{1,2}+\cdots+x_{n-1,n})$.  
However, with respect to this basis, the principal order~$\mathfrak{A}$ is not, in general, standard.

Set
\begin{align*}
\H^1&=(1+\mathfrak{p}_{\E})\U^{\lfloor m/2\rfloor +1}(\Lambda),\quad 
\J^1=(1+\mathfrak{p}_{\E})\U^{\lfloor (m+1)/2\rfloor}(\Lambda)\\
\J&=\mathfrak{o}_{\E}^\times \U^{\lfloor (m+1)/2\rfloor}(\Lambda),\quad\mathbf{J}=\E^\times \U^{\lfloor (m+1)/2\rfloor}(\Lambda).
\end{align*}
(The centres of these groups are equal to their intersections with~$\F^\times$.)

The set of \emph{simple characters} of~$\H^1$ associated to~$[\mathfrak{A},m,0,\beta]$ is
\[\mathcal{C}(\Lambda,\beta)=\{\theta\in \Hom(\H^1,\R^\times):\theta\mid_{\U^{\lfloor m/2\rfloor +1}(\Lambda)}=\psi_{\beta}\}.\]
As~$\E^\times$ normalizes~$\Lambda$,~$\U^{\lfloor m/2\rfloor +1}(\Lambda)$ is a normal subgroup of~$\H^1$.  %Set

Given~$\theta\in\mathcal{C}(\Lambda,\beta)$, we set
\[\eta_{\theta}=\ind_{(\U_n\cap \J^1)\H^1}^{\J^1}(\theta_{\psi})\]
where~$\theta_{\psi}(uh)=\Psi(u)\theta(h)$ (note that, $\U_n\cap \H^1=\U_n\cap \U^{\lfloor m/2\rfloor +1}(\Lambda)$ and~$\Psi$ and~$\theta$ clearly agree on this intersection).  The representation~$\eta_\theta$ is the unique irreducible representation of~$\J^1$ containing~$\theta$ (from the theory of Heisenberg representations).  Notice if~$m$ is odd, we have~$\H^1=\J^1$ and~$\eta_{\theta}=\theta$.    The representation~$\eta_{\theta}$ extends to~$\mathbf{J}$, choose any extension~${\boldsymbol\lambda}$ of~$\eta_{\theta}$, then
\[\pi=\ind_{\mathbf{J}}^{\G}({\boldsymbol\lambda})\]
is irreducible, minimax, and supercuspidal \emph{(and all minimax supercuspidals arise in this way).}   %By \cite{Bushnell}, , and Proposition \ref{Scconductor} in the modular setting,~$c(\pi)=n(1+d(\pi))$.  
The ramification index and inertial degree of~$\E/\F$ are invariants of~$\pi$ and we write~$e(\pi)=e(\E/\F)$ and~$f(\pi)=f(\E/\F)$.

\subsection{Newform vectors in minimax cuspidals}
Recall that~$\Sigma_n$ denotes the diagonal matrix~$\mathrm{diag}(\varpi_\F^{n-1},\varpi_\F^{n-2},\ldots,1)$.  We start with a basic lemma:
\begin{lemma}\label{Lemmapsibeta}
Let~$[\Lambda,m,0,\beta]$ be a simple stratum in~$\A$ such that~$[\Lambda,m,m-1,\beta]$ is a minimax stratum in~$\A$, and write~$\U^{\lfloor m/2\rfloor +1}(\Lambda)$ with respect to the companion basis of~$[\Lambda,m,0,\beta]$.  Then
\[\psi_{\beta}\mid_{\U^{\lfloor m/2\rfloor +1}(\Lambda)\cap \presuper{\Sigma_n} \K_n(n(1+m/e))}=1.\]
\end{lemma}

\begin{proof} For $x=(x_{i,j})\in \U^{\lfloor m/2\rfloor +1}(\Lambda)$ the character $\psi_\beta$ is given explicitly by
$$\psi_{\beta}(x)=\psi\left(\sum_{i=1}^{n-1}x_{i,i+1}-\sum_{i=1}^{n-1}a_{i-1} x_{n,i} -a_{n-1}(x_{n,n}-1)\right).$$
Suppose now that $x\in\presuper{{\Sigma_{n}}}\K_n(n(1+m/e))$ and note that
$$\presuper{{\Sigma_{n}}}\K_n(n(1+m/e))= \begin{pmatrix}\mathfrak{o} &\mathfrak{p} &\mathfrak{p}^2&\cdots&\\
	\mathfrak{p}^{-1}&\mathfrak{o}&\mathfrak{p} &\\
	&\ddots &\ddots&\ddots&\\\\\\
	\mathfrak{p}^{nm/e+1}&\mathfrak{p}^{nm/e+2}&\cdots &\mathfrak{p}^{nm/e+(n-1)}&1+\mathfrak{p}^{n(m/e+1)}
\end{pmatrix}.$$
Hence $x_{i,i+1}$ is an element of $\mathfrak p$ for $i=1,,\dotsc, n-1$ and 
$$\psi\left(\sum_{i=1}^{n-1}x_{i,i+1}\right)=1.$$
Moreover, since $\nu_\F(a_i)\geqslant -m(n-i)/e$ we obtain that $a_{i-1}x_{n,i}\in\mathfrak p$ for $i=1,\dotsc,n-2$ and $a_{n-1}(x_{n,n}-1)\in\mathfrak p$ which implies the result.
\end{proof}

To go further than this we specialize to minimax representations of integral depth.  Note that a minimax cuspidal~$\pi$ is of integral depth if and only if~$e(\pi)=1$.  Let~$\pi$ be a minimax cuspidal with~$e(\pi)=1$ and let~$[\Lambda,m,0,\beta]$ be a simple stratum contained in~$\pi$ with~$[\Lambda,m,m-1,\beta]$ a minimax stratum. Note that~$\pi$ being of integral depth is equivalent to the extension~$\E=\F[\beta]$ being unramified. We change basis from~$\mathcal{B}$ to
\begin{align*}
\mathcal{B}'=\{\varpi^{km}\beta^k:0\leqslant k\leqslant n-1\}.\end{align*}
With respect to this basis
\begin{align*}
\Lambda(i)&=\bigoplus \mathfrak{p}_\F^i ;\\
\mathfrak{A}(\Lambda)&=\mathrm{M}_{n}(\mathfrak{o}),\quad \mathfrak{P}(\Lambda)=\M_n(\mathfrak{p}_\F)\\
\beta&=\begin{pmatrix}  &&&-\varpi^{(n-1)m} a_0\\
\varpi^{-m}&&&-\varpi^{(n-2)m} a_1\\
&\ddots&&\vdots\\
&&\varpi^{-m}&-a_{n-1}\end{pmatrix}.
\end{align*}
(In particular,~$\U^r(\Lambda)=1+\mathrm{M}_n(\mathfrak{p}^r)$.)  

Moreover, in this basis, the character~$\psi_{\beta}$ is given explicitly by
\begin{equation*}%\label{psibeta}
\psi_{\beta}(x)=\psi\left(\left(\varpi^{-m}\sum_{i=1}^{n-1}x_{i,i+1}\right)-\left(\sum_{i=1}^{n-1}\varpi^{(n-i)m}a_{i-1} x_{n,i}\right) -a_{n-1}(x_{n,n}-1)\right)\end{equation*}
where~$x=(x_{ij})$, which on upper triangular unipotent matrices~$\U_n$ agrees with the non-degenerate character~$\psi^{\mathbf{t}_m}$ where~$\mathbf{t}_m=\diag(\varpi^{(n-1)m},\varpi^{(n-2)m},\ldots,1)$.

As explained in the last section, there is a simple character~$\theta\in\mathcal{C}(\Lambda,\beta)$, and an extension~$\boldsymbol\lambda$ of~$\eta_{\theta}$ to~$\mathbf{J}$ such that~$\pi\simeq \ind_{\mathbf{J}}^{\G}({\boldsymbol\lambda})$.  By Bushnell's formula, and Proposition \ref{Scconductor} in the modular setting,~$c(\pi)=n(m+1)$.  Let~
\[\Sigma_{m,n}=\mathrm{diag}(\varpi^{(m+1)(n-1)},\varpi^{(m+1)(n-2)},\ldots, \varpi^{m+1},1)=\mathbf{t}_m\Sigma_n,\] then
\[\presuper{{\Sigma_{m,n}}}\K_n(n(m+1))= \begin{pmatrix}\mathfrak{o} &\mathfrak{p}^{m+1} &\cdots&&\\
\mathfrak{p}^{-m-1}&\mathfrak{o}&\mathfrak{p}^{m+1} &\\
&\ddots &\ddots&\ddots&\\\\\\
\mathfrak{p}^{m+1}&\mathfrak{p}^{2(m+1)}&\cdots &\mathfrak{p}^{(n-1)(m+1)}&1+\mathfrak{p}^{n(m+1)}
\end{pmatrix}.\]

The aim of this section is to prove:
\begin{thm}\label{Propositionunramnewforms}
Let~$n\geqslant 2$ and~$\Lambda,\beta, \mathcal{B}'$ as above. Let~$\pi=\ind_{\mathbf{J}}^{\G_n}({\boldsymbol\lambda})$ be a unramified minimax cuspidal~$\R$-representation with cuspidal~$\R$-type~$(\J,\lambda)$, constructed using the simple stratum~$[\Lambda,m,0,\beta]$.
\begin{enumerate}
\item  Then~$c(\pi)=n(m+1)$ and~$\Hom_{\R[\K_n(c(\pi))]}(1,\pi)\simeq \R$.
\item The unique (up to scalar) non-zero newform~$f_{\new}\in\pi^{\K_n(n(m+1))}$ is given by 
\begin{align*}
\mathrm{supp}(f_{\new})&\subseteq\mathbf{J}\Sigma_{m,n}\K_n(n(m+1))\\
f_{\new}(j\Sigma_{m,n}k)&={\boldsymbol\lambda}(j)\cdot f_{\new}(\Sigma_{m,n}) \end{align*}
for all~$j\in\mathbf{J}$,~$k\in\K_n(n(m+1))$, and~$\mathrm{supp}(f_{\new}(\Sigma_{m,n}))=(\H^1(\J^1\cap\U_n))(\presuper{\Sigma_{m,n}}\K_n(n(m+1))\cap\J^1)$ with
\begin{align*}
f_{\new}(\Sigma_{m,n})(hk)=\theta_{\psi}(h)\end{align*}
for~$h\in\H^1(\J^1\cap\U_n)$, $k\in (\presuper{\Sigma_{m,n}}\K_n(n(m+1))\cap\J^1)$.
Moreover, there exists a unique~$\R$-Haar measure on~$\K_n(n(m+1))$ such that, for all~$g\in\G$,
\[f_{\new}(g)=\int_{\K_n(n(m+1))}\mathcal{J}(gk\Sigma_{m,n}^{-1})dk\]
where~$\mathcal{J}\in W({\boldsymbol\lambda},\theta_{\psi})\subseteq\ind_{(\U\cap \J^1)\H^1}^{\mathbf{J}}(\theta_{\psi})$ is the Paskunas--Stevens Bessel function.
\end{enumerate}
\end{thm}

We need the following basic lemma on intersections:

\begin{lemma}\label{groupintersectionlemma}
\begin{enumerate}
\item \label{intersection1}
$\H^1\cap \presuper{{\Sigma_{m,n}}}\K_n(n(m+1))=\U^{\lfloor m/2\rfloor +1}(\Lambda)\cap  \presuper{{\Sigma_{m,n}}}\K_n(n(m+1)).$
\item\label{intersection2} $(\U\cap\J^1)\H^1 \cap \presuper{{\Sigma_{m,n}}}\K_n(n(m+1))=(\U\cap \U^{\lfloor m+1/2\rfloor}(\Lambda))\U^{\lfloor m/2\rfloor +1}(\Lambda)\cap  \presuper{{\Sigma_{m,n}}}\K_n(n(m+1)).$
\item\label{intersection3} $\mathbf{J}\cap \presuper{{\Sigma_{m,n}}}\K_n(n(m+1))=\J^1\cap   \presuper{{\Sigma_{m,n}}}\K_n(n(m+1))$.
\end{enumerate}
\end{lemma}
{\begin{proof}
All these statements follow from the following one:  for all~$1\leqslant r\leqslant m+1$, for~$p\in\P_n, x\in\mathfrak{o}_{\E},$ and~$u\in\U^r(\Lambda)$
\[pxu\in  \presuper{{\Sigma_{m,n}}}\K_n(n(m+1))\Rightarrow x\in1+\mathfrak{p}_{\E}^{r},\]
where~$\P_n$ here denotes the standard mirabolic subgroup.  As $\varpi^m\beta+\mathfrak p_E$ generates~$k_\E/k_\F$ we have $\mathfrak o_\E=\mathfrak o_\F[\varpi^m\beta]$.   Note that $\varpi^m\beta\in \M_n(\mathfrak o_\F)$. Since $x\in\mathfrak{o}_{\E}$, we can write
\[x=\sum_{j=0}^{n-1}\alpha_j(\varpi^m\beta)^j,\]
where~$\alpha_j\in\mathfrak o_\F$. For $1\leqslant i\leqslant n-1$,  let $v_i$ denote the bottom row of $(\varpi^m\beta)^i$, then one can compute that 
\[v_i=(0,\dotsc,0,1,r_i^{i},\dotsc,r_1^{i}),\]
where $r_1^{i},\dotsc,r_i^{i}\in\mathfrak o_\F$.  Set $x=(x_{ij})$, then for $1\leqslant j\leqslant n$, we clearly have that 
\[x_{nj}=\alpha_{n-j}+\sum_{l=1}^{j-1}\alpha_{n-l}r_{n-j-1}^{n-l}.\]
We also write $u=(u_{ij})$, where $u_{ij}\in\delta_{i,j}+\mathfrak p_\F^{r}$.  Then the $(n,1)$-entry of $pxu$ is equal to
	\[\alpha_{n-1}u_{11}+\sum_{j=2}^{n}x_{nj}u_{j1}\in\mathfrak p^r_\F,\]
	and since $u_{11}\in\mathcal O^\times$ and $u_{j1}\in\mathfrak p_\F^{r}$ for $j\geqslant 2$, we immediately obtain that $\alpha_{n-1}\in\mathfrak p_\F^{r}$.  Suppose now we have shown that $\alpha_{n-1},\dotsc,\alpha_{n-k}\in\mathfrak p_\F^{r}$ for some $1\leqslant k\leqslant n-1$. Then the $(n,k+1)$ entry of $pxu$ is
	$$u_{k+1,k+1}\left(\alpha_{n-k-1}+\sum_{l=1}^k\alpha_{n-l}r_{n-l}^{n-k}\right)+\sum_{\substack{1\leqslant j\leqslant n\\ j\not=k+1}}x_{nj}u_{j,k+1}\in\delta_{n,k+1}+\mathfrak p_F^{(k+1)(m+1)},$$ 
	which immediately implies that $\alpha_{n-k-1}\in\mathfrak p_\F^{r}$. Inductively, we obtain that $\alpha_1,\dotsc,\alpha_{n-1}\in\mathfrak p_\F^{r}$, and applying the inductive step one more time~$\alpha_0\in 1+\mathfrak{p}_\F^r$.  Since $\mathfrak p_\E^{r}=\mathfrak p_\F^{r}\mathfrak o_\E$ this implies that $x\in 1+\mathfrak p_\E^{r}$.
\end{proof}	
}

\begin{proof}[Proof of Theorem \ref{Propositionunramnewforms}]
We have already seen~$c(\pi)=n(m+1)$ and one dimensionality over~$\R$ of characteristic zero follows from \cite{JPSS}.  
From Lemma \ref{Lemmapsibeta}
\[\psi_{\beta}\mid_{\U^{\lfloor m/2\rfloor +1}(\Lambda)\cap \presuper{{\Sigma_{m,n}}}\K_n(n(m+1))}=1.\]
By Lemma \ref{groupintersectionlemma} \eqref{intersection1}, it follows that for any~$\theta\in\mathcal{C}(\Lambda,\beta)$
\[\theta\mid_{\H^1\cap \presuper{{\Sigma_{m,n}}}\K_n(n(m+1))}=\psi_{\beta}\mid_{\U^{\lfloor m/2\rfloor +1}(\Lambda)\cap  \presuper{{\Sigma_{m,n}}}\K_n(n(m+1))}=1.\]

Next we claim that~$\theta_{\psi}\mid_{(\U_n\cap\J^1) \H^1\cap  \presuper{{\Sigma_{m,n}}}\K_n(n(m+1))}=1$.  Indeed, via Lemma \ref{groupintersectionlemma} \eqref{intersection2}, we can write~$x\in (\U_n\cap\J^1) \H^1\cap  \presuper{{\Sigma_{m,n}}}\K_n(n(m+1))$ as~$x=uu'$ with~$u\in(\U_n\cap \U^{\lfloor m+1/2\rfloor}(\Lambda))$, and~$u'\in\U^{\lfloor m/2\rfloor +1}(\Lambda)$.  Moreover, the final row of $u'$ is contained in~$\left(\begin{smallmatrix} \mathfrak{p}_{\F}^{m+1}\cdots \mathfrak{p}_\F^{(n-1)(m+1)} (1+\mathfrak{p}_\F^{n(m+1)})\end{smallmatrix}\right)$ hence~$\psi_{\beta}(u')=\psi(\varpi^{-m}\sum_{i=1}^{n-1}u'_{i,i+1})$ and it suffices to notice that~$u_{i,i+1}+u'_{i,i+1}\in\mathfrak{p}_\F^{m+1}$. % \rob{explanation to add.}. 
This completes the claim.

By Mackey theory and Frobenius reciprocity, we have an embedding
\begin{equation*}%\label{injection}
\Phi:\R=\Hom_{(\U_n\cap\J^1)\H^1\cap  \presuper{{\Sigma_{m,n}}}\K_n(n(m+1))}(1,\theta_{\psi})\hookrightarrow \Hom_{\J^1\cap  \presuper{{\Sigma_{m,n}}}\K_n(n(m+1))}(1,\ind_{(\U_n\cap\J^1) \H^1}^{\J^1}(\theta_{\psi})).\end{equation*}
Moreover, from Lemma \ref{groupintersectionlemma} \eqref{intersection3}, we have
\[\Hom_{\mathbf{J}\cap  \presuper{{\Sigma_{m,n}}}\K_n(n(m+1))}(1,{\boldsymbol\lambda})=\Hom_{\J^1\cap  \presuper{{\Sigma_{m,n}}}\K_n(n(m+1))}(1,\ind_{(\U_n\cap\J^1) \H^1}^{\J^1}(\theta_{\psi})),\]
and it follows from one dimensionality over algebraically closed fields of characteristic zero that $\Phi$ is an isomorphism.   

Moreover, it follows from reduction mod~$\ell$,  that
\[\Hom_{\J^1\cap  \presuper{{\Sigma_{m,n}}}\K_n(n(m+1))}(1,\ind_{(\U_n\cap\J^1) \H^1}^{\J^1}(\theta_{\psi}))=\R\] for all algebraically closed fields (as it is a Hom-space over a pro-$p$ group and~$\ell\neq p$), which allows us to deduce the one dimensionality %``uniqueness of newforms'' 
in positive characteristic too (as all lifts are contributing the same Hom-space, this is the unique one mod~$\ell$ and has dimension one by the previous lifting argument).  The explicit form of the vector~$f_{\new}$ follows from reversing the Mackey theory.

For the final integral formula, as it is clearly~$\K_n(n(m+1))$-invariant, it suffices to show it is non-zero.  Evaluating the integral at~$\Sigma_{m,n}$ we obtain
\[\int_{\K_n(n(m+1))} \mathcal{J}(\Sigma_{m,n}k\Sigma_{m,n}^{-1})dk=\int_{\J^1
\cap\K_n(n(m+1))} \mathcal{J}(\Sigma_{m,n}k\Sigma_{m,n}^{-1})dk,\]
by Lemma \ref{groupintersectionlemma} \eqref{intersection2}.  Writing 
\[k=\left( \begin{smallmatrix} A& \underline{u}\\ k_{n,1}\cdots k_{n,n-1}&1+k_{n,n}\end{smallmatrix}\right)\]
and setting $k_P=\left( \begin{smallmatrix} A& \underline{u}\\ 0\cdots0&1\end{smallmatrix}\right)$, by \cite[Proposition 5.3 (iii) \& (iv)]{PaskSte} we have
\begin{align*}
\mathcal{J}(\Sigma_{m,n}k\Sigma_{m,n}^{-1})=\mathcal{J}(\Sigma_{m,n}k_P\Sigma_{m,n}^{-1})=
\begin{cases}
\theta_{\psi}(\Sigma_{m,n}k_P\Sigma_{m,n}^{-1}),&\text{if }\Sigma_{m,n}k_P\Sigma_{m,n}^{-1}\in (\U_n\cap\J^1)\H^1\\
0&\text{otherwise.}\end{cases}
\end{align*}
Moreover, $\mathcal{J}\mid_{(\U_n\cap\J^1) \H^1\cap  \presuper{{\Sigma_{m,n}}}\K_n(n(m+1))}$ is trivial as it is a multiple of~$\theta_{\psi}$ which is trivial when restricted to this subgroup, so the integral is a volume of a pro-$p$ group and we are done.
\end{proof}

\begin{remark}
The key properties needed for the proof are
\begin{enumerate}
\item  the restriction of~$\theta_{\psi}$ to the intersection of its domain with~$\presuper{\Sigma}\K_n(c(\pi))$ is trivial;
\item  $\J\cap \presuper{\Sigma}\K_n(c(\pi))=\J^1\cap\presuper{\Sigma}\K_n(c(\pi))$.
\end{enumerate}
Currently we are missing a construction of such an element~$\Sigma$ in the general minimax case, although Lemma 
\ref{Lemmapsibeta} suggests a candidate (there in the companion matrix basis).
\end{remark}

\begin{cor}\label{WhittcorollaryURminimax}
Let~$n\geqslant 2$ and~$\Lambda,\beta, \mathcal{B}'$ as above. Let~$\pi=\ind_{\mathbf{J}}^{\G_n}({\boldsymbol\lambda})$ be an unramified minimax cuspidal~$\R$-representation with cuspidal~$\R$-type~$(\J,\lambda)$, constructed using the simple stratum~$[\Lambda,m,0,\beta]$ and fixed character~$\psi:\F\rightarrow \mathbb{C}^\times$ of conductor $\mathfrak{p}_\F$.  Then~$\psi^{\Sigma_n}$ has conductor~$\mathfrak{o}_\F$, and
\[\W_{\pi,\new,\psi^{\Sigma_n}}(g)=\int_{\K_n(n(m+1))}\W_{\pi,\mathrm{PS},\psi^{\mathbf{t}_m}}(\Sigma_{m,n} gk \Sigma_{m,n}^{-1}) dk,\]
where~$\W_{\pi,\mathrm{PS},\psi^{\mathbf{t}_m}}\in W(\pi,\psi^{\mathbf{t}_m})$ is the Paskunas--Stevens Whittaker function of \cite[Theorem 5.8]{PaskSte} and \cite[Section 7]{ALSX}.
\end{cor}

\bibliographystyle{alpha}
\bibliography{newforms}
\end{document}